\documentclass[
leqno, 
12pt]{amsart}
\usepackage{amssymb,amsmath, amsthm,latexsym, mathrsfs}

\usepackage{dashbox}

\usepackage{a4wide}

\usepackage{hyperref}

\usepackage{rotating}

\theoremstyle{plain}

\newtheorem{lemma}{Lemma}[section]
\newtheorem{theorem}[lemma]{Theorem}
\newtheorem{proposition}[lemma]{Proposition} 
\newtheorem{corollary}[lemma]{Corollary}

\newtheorem{remark}[lemma]{Remark}

\newtheorem{definition}[lemma]{Definition}

\parskip=\bigskipamount

\def\Z{\mathbb Z}
\def\R{\mathbb R}
\def\d{\delta}
\def\r{\rtimes}
\def\t{\times}
\def\o{\otimes}
\def\e{\epsilon}

\def\h{\hookrightarrow}

\def\a{\alpha}
\def\b{\beta}
\def\g{\gamma}

\def\D{\Delta}
\def\G{\Gamma}

\def\s{\sigma}

\def\supp{\text{supp} }

\def\BE{\begin{equation}}
\def\EE{\end{equation}}

\title[On unitarizability]
{On unitarizability in the case of classical $p$-adic groups}

\author[M. Tadi\'c]
{Marko Tadi\'c}

\address{Department of Mathematics, University of Zagreb
\\
Bijeni\v{c}ka 30, 10000 Zagreb,
 Croatia\\
Email: \tt tadic{\char'100}math.hr}

\keywords{non-archimedean local fields, classical $p$-adic groups, irreducible representations, unitarizability, parabolic induction}
\subjclass[2000]{Primary: 22E50}

\thanks{This work has been  supported  by Croatian Science Foundation under the
project 9364.}
\date{\today}

\begin{document}

\begin{abstract} In the introduction of this paper we discuss a possible approach to the unitarizability problem for classical $p$-adic groups. In this paper we give some very limited  support that such approach is not without chance. In a forthcoming paper we shall give additional evidence in generalized cuspidal rank (up to) three.

\end{abstract}

\maketitle

\setcounter{tocdepth}{1}

%\tableofcontents

\section{Introduction}
\label{intro}
\setcounter{equation}{0}

Some important classes of irreducible unitary representations of classical $p$-adic groups have been classified. Still,  classification  of the whole unitary dual of these groups does not seem to be in sight in the moment.
Since the case of general linear groups is well-understood, we shall start with description of the unitarizability in the case of these groups, the history related to this and what this case could suggest us regarding the unitarizability for the classical $p$-adic groups.

Fix a local  
field $F$. Denote by $GL(n,F)\widehat{\ }$ the set of all equivalence classes of irreducible unitary representations of $GL(n,F)$.
We shall use a well-known notation $\t$ of Bernstein and Zelevinsky for parabolic induction 
of two representations $\pi_i$ of $GL(n_i,F)$:
 $$
 \pi_1\t\pi_2=\text{Ind}
 ^{GL(n_1+n_2,F)}(\pi_1\o\pi_2)
 $$
 (the above representation is parabolically induced  from a suitable parabolic subgroup containing upper triangular matrices whose Levi factor is naturally isomorphic to the direct product $GL(n_1,F)\t GL(n_2,F)$).
Denote by $\nu$ the character $|\det|_F$ of a general linear group. Let
 $
  D_u=D_u(F)
  $
 be the set of all the equivalence classes of the irreducible square integrable (modulo center) representations of all $GL(n,F)$, $n\geq 1$. 
 For $\d\in D_u$   and $m\geq 1$ denote by
 $$
 u(\d,m)
 $$
 the unique irreducible quotient of 
 $
 \nu^{(m-1)/2}\d\t\nu^{(m-1)/2-1}\d\t\dots \t\nu^{-(m-1)/2}\d.
 $
 This irreducible quotient is called a Speh representation. Let
 $
 B_{rigid}
 $ 
 be the set of all Speh representations, and
 $$
 B=B(F)=B_{rigid}\cup \{\nu^\a\s\t\nu^{-\a}\s;\s\in B_{rigid}, 0<\a<1/2\}.
 $$
 Denote by $M(B)$ the set of all finite multisets in $B$.
 Then the following simple theorem solves the unitarizability problem for the  archimedean and the non-archimedean general linear groups in a uniform way:
 
 \begin{theorem} 
 \label{th-1-int}
 $($\cite{T-AENS}, \cite{T-R-C-new}$)$
  A mapping 
 $$
 (\s_1,\dots,\s_k)\mapsto \s_1\t\dots\t\s_k
 $$
 defined on $M(B)$ goes into 
 $
 \cup_{n\geq0}GL(n,F)\widehat{\ },
 $
 and it is a bijection.
\end{theorem} 

The above theorem was first proved in the $p$-adic case in mid 1980-es (in \cite{T-AENS}). Since the claim of the theorem makes sense also in the archimedean case,  immediately became evident that the theorem extends also to the archimedean case, with the same strategy of the proof (the main ingredients of the proof were already present in that time, although one of them was announced by Kirillov, but the proof was not complete in that time). One can easily get an idea of the proof from \cite{T-Bul-naj} (there is  considered the $p$-adic case, but exactly  the same strategy holds in the archimedean case). Vogan's classification in the archimedean case (Theorem 6.18 of \cite{Vo}) gives a very different description of the unitary dual (it is equivalent to Theorem \ref{th-1-int}, but it is not obvious to see that it is equivalent - see section 12 of \cite{BR-Comp}).

In the rest of this paper, we shall consider only the case of non-archimedean field $F$. Although the representation theory of reductive $p$-adic groups started with the F. Mautner paper \cite{Mt} from 1958, the ideas that lead to the proof of the above theorem can be traced back  to the paper of I. M. Gelfand and M. A. Naimark  \cite{GN47} from 1947, and together with the work on the unitarizability of general linear groups over division algebras, we may say that spans a period of  seven decades.

The proof of Theorem \ref{th-1-int} in \cite{T-AENS} is based on a very subtle  Bernstein-Zelevinsky theory based on the  derivatives (\cite{Z}), and on the Bernstein's paper \cite{Be-P-inv}. 
 Among others , the  Bernstein's paper \cite{Be-P-inv}  proves a fundamental fact about distributions on general linear groups. It   is based on the geometry of these groups (a key idea of that paper 
 can be traced back  to the Kirillov's  paper \cite{Ki} from 1962, which is motivated by a result of the Gelfand-Naimark book \cite{GN}). One of the approaches to the unitarizability of the Speh representations is using the H. Jacquet's construction of the residual spectrum of the spaces of the square integrable automorphic forms in \cite{Jq}, which generalizes an earlier construction of B. Speh in \cite{Sp}.

We presented in \cite{T-div-a}   what we expected to be the answer to the unitarizability  question  for general linear groups over a local non-archimedean division algebra $\mathcal A$\footnote{For $\d\in D(\mathcal A)_u$ denote by $\nu_\d:=\nu^{s_\d}$, where $s_\d$ is the smallest non-negative number such that $\nu^{s_\d}\d\t\d$ reduces. Introduce $u(\d,n)$ in the same way as above, except that we use $\nu_\d$ in the definition of $u(\d,n)$ instead of $\nu$. Then the expected answer  is the same as in the Theorem \ref{th-1-int}, except that one replaces $\nu$ by $\nu_\d$ in the definition of $B(\mathcal A)$.}.
We  have  reduced in \cite{T-div-a} a proof of the expected answer to  two expected facts. 
They were proved by J. Badulescu and D. Renard (\cite{BR-Tad}), and V. S\'echeree (\cite{Se}). As well as in the field case, these proofs (together with the theories that they require) 
are far from being simple (the S\'echeree proof is particularly technically complicated since it requires knowledge of a 
complete  theory of types for these groups).

As a kind of surprise   came a recent work \cite{LpMi} of
 E. Lapid and A. M\'inguez  in which they gave 
 another  (surprisingly simple in comparison  with the earlier) proof of the S\'echeree result (relaying on the Jacquet module methods). 
Besides, J. Badulescu gave earlier in \cite{Bad-Speh} another very simple (local) proof of his and Renard's result.

Thanks to this new development, we have a pretty simple approach  to the unitarizability for general linear groups over non-archimedean division algebras, using only
  very standard non-unitary theory and knowledge of   the reducibility point between two irreducible cuspidal representations of general linear groups, i.e. when $\rho\t\rho'$ reduces for $\rho$ and $\rho'$ irreducible cuspidal representations\footnote{In the field case it reduces if and only if $\rho'\cong\nu^{\pm1}\rho$.}.
It is very important that we have such a relative simple approach to the irreducible unitary representations in this case, since
these representations 
  are basic ingredients of some very important unitary representations, like the representations in the spaces of the square integrable automorphic forms, and their knowledge can be quite useful 
(see \cite{LaRpSt}, or \cite{HTy} or \cite{He})

Thanks to the work of J. Arthur, C. M\oe glin and J.-L. Waldspurger, 
 we have now classification of the irreducible cuspidal representations of classical $p$-adic groups\footnote{By classical groups we mean symplectic, orthogonal and unitary  groups (see the following sections for more details). In this introduction and in the most of the paper we shall deal with symplectic and orthogonal groups. The case of unitary groups is discussed in the last section of the paper.} in the characteristic zero (Theorem 1.5.1 of \cite{Moe-mult-1} and Corollary 3.5 of \cite{MoeR}). Their parameters give directly the reducibility points with irreducible cuspidal representations of general linear groups (see for example (ii) in Remarks 4.5.2 of \cite{MW} among other papers). 
These reducibilities are any of $0,\frac12,1,\frac32,2,\dots$. 

Therefore, a natural and important question is if we can have an approach to the unitarizability in the case of   classical $p$-adic groups
based only on the cuspidal reducibility points.
We shall try to explain a possible strategy for such an approach based only on the reducibility points.

We fix a series of classical $p$-adic groups (see the section \ref{notation} for more details).
First we shall introduce a notation for parabolic induction for the classical $p$-adic groups.
The multiplication $\t$ between representations of general linear groups defined using parabolic induction has a natural generalization  to a multiplication 
$$
\r
$$
  between representations of  general linear and  classical groups defined again using the parabolic induction (see the second section of this paper).

Now we shall recall of the reduction of the unitarizability problem obtained in \cite{T-CJM}.  An irreducible representation $\pi$ of a classical group is called weakly real if there exist self contragredient  irreducible cuspidal representations $\rho_i$ of general linear groups, an irreducible cuspidal representation $\s$ of a classical group, and $x_i\in\R$, such that
$$
\pi\h \nu^{x_1}\rho_1\t\dots\t \nu^{x_k}\rho_k\r\s.
$$
Then \cite{T-CJM} reduces in a simple way the unitarizability problem for the  classical $p$-adic groups to the case of the weakly real representations of that series of groups (see Theorems \ref{unitary-red-1} and \ref{unitary-red-2} of this paper). This is the reason that we shall consider only the weakly real representations in the sequel.

Let $X$ be some set of irreducible cuspidal representations of the general linear groups. For an irreducible representation $\tau$  of a general linear group one says that it is supported by $X$ if there exist $\rho_i\in X$ such that
$\tau\h \rho_1\t\dots\t \rho_k$. Let additionally $\s$ be an  irreducible cuspidal representations of a classical $p$-adic group and assume $X=\tilde X:=\{\tilde\rho; \rho\in X\}$ ($\tilde\rho$ denotes the contragredient representation of $\rho$). Then for an irreducible representation $\pi$  of a classical $p$-adic group one says that it is supported by $X\cup \{\s\}$ if there exist $\rho_i\in X$ such that
$\pi\h \rho_1\t\dots\t \rho_k\r\s$. In that case we say that 
$$
\s
$$
 is a partial cuspidal support of $\pi$.

  Let $\rho$ be an irreducible self contragredient cuspidal representation of a general linear group. Denote $X_\rho:=\{\nu^x\rho; x\in \R\}$. We call $X_\rho$ a line  of cuspidal representations. Further, denote by
  $$
  Irr(X_\rho;\s)
  $$ 
  the set of all equivalence classes of irreducible representations of classical groups supported by $X_\rho\cup\{\s\}$.
  
  Let $\pi$ be an irreducible (weakly real) representation of a classical $p$-adic group and  denote its partial cuspidal support by $\s$.
  Then for $\rho$ as above,  there exists a unique irreducible representation
$X_\rho(\pi)$  of a classical group supported in $X_\rho \cup\{\s\}$ and an irreducible representation 
$X_\rho^c(\pi)$ of a general linear group supported out of $X_\rho$ such that 
\begin{equation}
\label{cl}
\pi\h X_\rho^c(\pi)\r X_\rho(\pi).
\end{equation}
One can chose a finite set $\rho_1,\dots,\rho_k$ of irreducible selfcontragredient cuspidal representations of  general linear groups such that for other selfcontragredient representations $\rho$ of  general linear groups we have always $X_\rho(\pi)=\s$.

Fix some set  $\rho_1,\dots,\rho_k$ as above. Then
C. Jantzen has shown in \cite{J-supp} that  the correspondence
\begin{equation}
\label{????}
\pi \leftrightarrow (X_{\rho_1}(\pi),\dots,X_{\rho_k}(\pi))
\end{equation}
is a bijection from the set of all irreducible representations of classical groups supported by $X_{\rho_1},\dots,X_{\rho_k}\cup\{\s\}$ onto the direct product 
 $\prod_{i=1}^k Irr(X_{\rho_i};\s)$. 
Moreover, C. Jantzen has shown that the above correspondence reduces  some of the most basic data from the non-unitary theory about general parabolically induced representations (like for example the Kazhdan-Lusztig multiplicities) to the corresponding data for such  representations supported by  single cuspidal lines.

Regarding the unitarizability, 
 it would be very important  to know the answer to the following
 
 {\bf Preservation of unitarizability question:}
Is $\pi$ is unitarizable if and only if all $X_{\rho_i}(\pi)$ are unitarizable, $i=1,\dots,k.$

If we would know that the answer to the above question is positive,   then this would give a  
reduction of the unitarizability of a general irreducible representation  to the unitarizability for the irreducible representations of classical $p$-adic groups supported in single cuspidal lines. Such a line $X_\rho\cup\{\s\}$ is determined  by $\rho$ and $\s$, for which there exists a unique   $\alpha_{\rho,\sigma}\geq 0$ such that
$$
\nu^{\alpha_{\rho,\sigma}}\rho\r\s
$$
reduces.
In this paper we shall consider only the cases when
\begin{equation}
\label{1/2Z}
\alpha_{\rho,\sigma}\in (1/2)\mathbb Z.
\end{equation}
Actually, from the recent work of J. Arthur, C. M\oe glin and J.-L. Waldspurger, this assumption is known to hold if char$(F)=0$ (Theorem 3.1.1 of \cite{Moe-paq-stab}, and \cite{MoeR}).

Now suppose that we have  additional pair $\rho',\s'$ as $\rho,\s$. Assume that
$$
\alpha_{\rho,\sigma}=\alpha_{\rho',\sigma'}.
$$
If $\alpha_{\rho,\sigma}>0$, then there exists a canonical bijection
$$
E:Irr(X_\rho,\s)\rightarrow Irr(X_{\rho'};\s')
$$
(see the section \ref{independence} for the definition\footnote{It is there denoted by $E_{1,2}$.} of $E$). If $\alpha_{\rho,\sigma}=0$, then we have two possibilities for such a bijection (see again the section \ref{independence}). Chose any one of them and denote it  again by $E$. 

As we already mentioned, the parameter which determines the set $Irr(X_\rho,\s)$ (whose unitarizable representations we would like to determine) is  the pair $\rho,\s$. This pair  determines the cuspidal  reducibility point $\a_{\rho,\s}\in\frac12\Z$, which is a very simple object in comparison with the pair $\rho,\s$. Therefore, a natural question would be to try to see if the unitarizability depends only on $\a_{\rho,\s}$, and not on $\rho,\s$ itself. More precisely, we have the following

{\bf Independence of unitarizability  question:} Let $\pi\in Irr(X_\rho,\s)$. Does it hold that $\pi$ is unitarizable if and only if $E(\pi)$ is unitarizable.

In this paper we give some  very limited evidence that  one could expect positive answers to the above two questions. Using the classification of the generic unitary duals in \cite{LMuT}, we get that the both above  questions have positive answers in the case of irreducible generic representations (see the section \ref{generic}). Also the classification of the unramified   unitary duals in \cite{MuT} implies that we have  positive answer to the  first question in the case of irreducible unramified representations.

Further, very limited evidence for positive answer to the second question  
 give papers \cite{HTd} and \cite{HJ}. They imply that Independence  of unitarizability   question  has positive answer for irreducible representations which have the same infinitesimal character as a generalized Steinberg representation\footnote{Generalized Steinberg representations are defined and studied in \cite{T-AJM}. See the section \ref{SQ} of this paper for a definition}.  The biggest part of this paper is related to Preservation of unitarizability question  for representation whose one Jantzen component $X_\rho(\pi)$ has  the same infinitesimal character like a generalized Steinberg representation. We are able to prove the following very special case  related to Preservation of unitarizability question  for such representations:

\begin{theorem}
\label{th-main-intro} Suppose that $\pi$ is an irreducible unitarizable representation of a classical group, and suppose that the  infinitesimal character of some $X_\rho(\pi)$ is the same as the infinitesimal character of a generalized Steinberg representation 
corresponding to a reducibility point 
$\a_{\rho,\s}\in\frac12\Z$\footnote{As we already noted, this is know to hold if char$(F)=0$.}. Then $X_\rho(\pi)$ is 
 the generalized Steinberg representation, or its  Aubert-Schneider-Stuhler  dual.
 If char$(F)=0$, then $X_\rho(\pi)$ is unitarizable.
\end{theorem}

 For the case when $\pi= X_\rho(\pi)$ (i.e. when $\pi$ is supported by a single cuspidal line), 
 in \cite{HTd} and \cite{HJ} is proved that $\pi$ is a generalized Steinberg representation or its  Aubert-Schneider-Stuhler  dual (which are both unitarizable in characteristic zero). Our first idea  to prove the above theorem (more precisely, to prove that $X_\rho(\pi)$ is a generalized Steinberg representation or its  Aubert-Schneider-Stuhler  dual) was to use the strategy of that two papers and  the methods of \cite{J-supp}. While we were successful in extending \cite{HTd}, we were not for \cite{HJ}. This was a reason for a search of  
 a new (uniform) proof for \cite{HTd} and \cite{HJ}, which   is   easy to extend to the proof  of the above theorem (using \cite{J-supp}).
  This new proof is based on the following fact.

 \begin{proposition}
\label{prop-main-intro}
Fix  irreducible   cuspidal representations $\rho$ and $\s$  of a general linear and a classical group respectively, such that $\rho$ is self contragredient. Suppose that
 $
 \nu^\a\rho\r\s
 $
 reduces for some $\a\in\frac12 \Z$, $\a>0$.
 Let  $\gamma$ be an irreducible  subquotient   of
\begin{equation}
\label{first eq}
\nu^{\alpha+n}\rho\times\nu^{\alpha+n-1}\rho\times\cdots\times\nu^{\alpha}\rho\rtimes \sigma,
\end{equation}
  different from the generalized Steinberg representation and its  Aubert-Schneider-Stuhler  involution\footnote{The generalized Steinberg representation is a unique irreducible subrepresentation of \eqref{first eq}, while  its  Aubert-Schneider-Stuhler  involution is the unique irreducible quotient of \eqref{first eq}.}.
Then there exists an irreducible selfcontragredient unitarizable representation $\tau$ of a general linear group with support  in $\{\nu^k\rho;k\in\frac12\Z\}$, such that the length of 
$$
\tau\r\gamma
$$
is at least 5, and that the multiplicity of
$
\tau\o\gamma$ in the Jacquet module of $\tau\r\gamma$ is at most 4.

\end{proposition}
A bigger part of this paper is the proof of the above proposition. The proof  is pretty technical and we shall say only a few remarks about it here. We often use  Proposition 5.3 of \cite{T-CJM}\footnote{This is an extension to the case of classical groups   of Proposition 8.4 of \cite{Z}, which in the terms of the Langlands classification tells that $L(a+b)\leq L(a)\t L(b)$ (see the section \ref{notation} for notation).}   in proving that the length of $\tau\r\g$ is at lest five. Two irreducible sub quotients we get directly applying this proposition. For remaining irreducible sub quotients we consider $\tau\r\g$ as a part of a bigger representation $\Pi$ which has the same semi simplification as $\tau\r\g+\Pi'$, for some $\Pi'$. An advantage of $\Pi$ in comparison to $\tau\r\g$ is that we can easily write some irreducible sub quotients of it (using Proposition 5.3 of \cite{T-CJM}). Next step is to show that these irreducible sub quotients are not sub quotients of $\Pi'$. For this is particularly useful the  Geometric lemma, which is systematically applied through the structure of a twisted Hopf module which exists on the sum of the Grothendieck groups of the categories of the  finite length representations of the classical groups. 
Further, the multiplicity of $\tau\o\g$ in the Jacquet module of $\tau\r\g$ is estimated using the combinatorics which provides the above  structure of the twisted Hopf module.

  Now we shall  recall a little bit of the history of the unitarizability of the irreducible representations which have the same   infinitesimal character as a generalized Steinberg representation. The first case is the case of the Steinberg representations.  The question of their unitarizability in this case came from the question of cohomologically non-trivial irreducible unitary representations.  Their  non-unitarizability or unitarizability  was proved by W. Casselman (\cite{C-non-u}). His proof of the non-unitarizability relays   on the study of the Iwahori Hecke algebra. The importance of this non-unitarizability is very useful in considerations of the unitarizability in low ranks, since it implies also the non-existence of  complementary series which would end by the trivial representation (it also reproves the classical result of Kazhdan from \cite{Ka} in the $p$-adic case).
  
  A. Borel and N. Wallach observed that the Casselman's non-unitarizability  follows from the Howe-Moore theorem  about asymptotics of the matrix coefficients of the irreducible unitary representations (\cite{HM}) and the Casselman's asymptotics of the matrix coefficients of the admissible representations of  reductive $p$-adic groups (\cite{C-int}). Neither of that two methods can be used for the case of the generalized Steinberg representation. This was a motivation to write papers \cite{HTd} and \cite{HJ}. The strategy of the proofs of that two papers was for a $\g$ from Proposition \ref{prop-main-intro} to find  an irreducible unitarizable representation $\tau$ of a general linear group    such that $\tau\r\g $  is not semisimple. The semisimplicity of  $\tau\r\g $ (using the Frobenius reciprocity) would  imply that $\tau\o\g $ is in  the Jacquet module of each irreducible subquotient $\theta$ of $\tau\r\g $. 
  In \cite{HTd} and \cite{HJ}, there were found $\tau$ and $\theta$ such that $\tau\o\g $ is not in the Jacquet module of $\theta$.
  This implied the non-unitarizability of $\g$.

  In this paper  our strategy is to find $\tau$ such that the length of $\tau\r\g$ is strictly bigger then the multiplicity of $\tau\o\s$ in the Jacquet module of $\tau\r\s$ (the above proposition implies this).

    We are particularly thankful to C. Jantzen for reading  the  section \ref{Jantzen} of this paper, and giving  suggestions about it (in that section  are presented the main results of C. Jantzen from \cite{J-supp} in a  slightly reformulated form).
    We are very thankful 
to C. M\oe glin for her explanations regarding references related to some assumptions considered in this paper.
  We are also thankful to M. Hanzer, E. Lapid and A. Moy for useful discussions during the writing of this paper, and to the referee for  useful suggestions.

We are also thankful to the Simons Foundation for its generous travel and the local support during the Simons Symposium.

In a forthcoming paper we shall  prove  that the above two questions have positive answers for irreducible (weakly real) subquotients of
representations
$$
\rho_1\t\dots\t\rho_k\r\s, \qquad k\leq 3,
$$
where $\rho_i$ are irreducible cuspidal representations of general linear groups and $\s$ is an irreducible cuspidal representations of a classical $p$-adic group. Moreover, we shall classify such subquotients.
  
We shall now briefly review the contents of the paper.  The second section brings the notation that we use in the paper, while the third one describes the irreducible representations that we shall consider. 
The fourth section recalls of Proposition \ref{prop-main-intro} and explains what are the first two stages of its proof.
 In the fifth section 
 is the first stage of the proof (when the essentially square integrable representation of a general linear group with the lowest exponent that enters the Langlands parameter of $\g$ is non-cuspidal, and the tempered representation of the classical group which enters the Langlands parameter  of $\g$  is cuspidal). The following section considers  the situation as in the previous section, 
 except 
 that the essentially square integrable representation of a general linear group with the lowest exponent that enters the Langlands parameter  of $\g$  is now cuspidal. 
 Actually, we could handle these two cases  
 as a single case.
 Nevertheless, we split it, since the first  case is  simpler, and it  is convenient to consider it first. The seventh section handles the remaining case,
 when the tempered representation of a classical group which enters the Langlands parameter  of $\g$  is not cuspidal. This case  is obtained from the previous two sections by a simple application of the   Aubert-Schneider-Stuhler  involution. 
  At the end of this section we get the main results of \cite{HTd} and \cite{HJ} as a simple application of Proposition \ref{prop-main-intro}.
 In the following section we          recall of the Jantzen decomposition of an irreducible representation of a classical $p$-adic  group in a slightly modified version then in \cite{J-supp}, while the ninth  section  discusses the decomposition into the cuspidal lines. In the tenth section we give a proof of  Theorem \ref{th-1-int}, while in the following  section we show that the unitarizability is preserved in the case of the irreducible 
generic representations of classical $p$-adic groups. In a similar way, using \cite{MuT}, one can see also that  the  unitarizability is preserved for the  irreducible unramified representations of  the classical groups considered in \cite{MuT} (i.e. for the split classical $p$-adic groups). In the twelfth  section we formulate a question if the unitarizability for the irreducible representations of classical groups supported by a single cuspidal line depends only on the reducibility point (i.e., not on the particular cuspidal representations which have that reducibility). The last section discusses the case of unitary 
  groups.

\section{Notation and Preliminaries}
\label{notation}
\setcounter{equation}{0}

Now we shall briefly introduce the notation that we shall  use in the paper. One can find more details in \cite{T-Str} and \cite{Moe-T}.

We fix a local non-archimedean field $F$ of characteristic different from two. We denote by   $|\ \ |_F$ the normalized absolute value on $F$.

For the group $\mathcal G$ of $F$-rational points of a connected  reductive 
group 
over $F$, we  denote by $\mathcal R(
\mathcal
G)$ the Grothendieck group of the category Alg$_{f.l.}(\mathcal  
G)$
of all smooth representations of $\mathcal  G$ of finite length. We
denote by 
$$\text{s.s.}$$
 the semi simplification map  Alg$_{f.l.}(\mathcal  G)\rightarrow\mathcal R( \mathcal  G)$.  The  irreducible representations of $\mathcal G$ are also considered as elements of $\mathcal R(\mathcal G)$.

We have a natural ordering 
$
\leq
$
 on $\mathcal R(\mathcal  G)$ determined by the cone s.s.(Alg$_{f.l.}(\mathcal{G} )$).

If   $\text{s.s.}(\pi_1)\leq \text{s.s.}(\pi_2)$ for $\pi_i\in$ Alg$_{f.l.}(\mathcal{G} )$, then we
write simply $\pi_1\leq\pi_2$.

Now we go to the notation of the representation theory of
general linear groups (over $F$), following  the  standard notation of the Bernstein-Zelevinsky theory (\cite{Z}).
Denote
$$
\nu:GL(n,F) \to \mathbb R^\times, \ \ \nu(g) = |\text{det}(g)|_F.
$$
The set of equivalence classes of all irreducible essentially square integrable modulo center\footnote{These are irreducible representations which become square integrable modulo center after twist by a (not necessarily unitary) character of the group.} representations of all $GL(n,F)$, $n\geq 1$, is denoted by 
$$
D.
$$
For  $\delta\in D$
there exists a unique
$e(\delta)\in
\mathbb R$ and a unique  unitarizable  representation $\delta^u$ (which is  square integrable modulo center), such that
$$
\delta\cong\nu^{e(\delta)}\delta^u.
$$
The subset of cuspidal representations in $D$ is denoted by
$$
\mathcal C.
$$

For smooth
representations  $\pi_1$ and $\pi_2$   of $GL(n_1,F)$ and $GL(n_2,F)$ respectively,
$\pi_1\times\pi_2$ denotes the smooth representation of $GL(n_1+n_2,F)$
parabolically induced by $\pi_1\otimes\pi_2$ from the appropriate  maximal
standard  parabolic subgroup (for us, the standard parabolic subgroups will be those parabolic subgroups which contain the subgroup of the
upper triangular matrices). We use the normalized parabolic induction  in the paper.

We consider
$$
R=
\underset
{n\geq 0}
{\oplus}
 \mathcal R(GL(n,F))
$$
as  a graded group.  The parabolic induction   $\times$ lifts  naturally
to a $\mathbb Z$-bilinear mapping 
$  R\times R\to R$, which we denote again by $\times$. This $\mathbb Z$-bilinear
mapping factors through the tensor product, and the factoring   homomorphism  is denoted by $m: R\otimes R\to R$. 

Let $\pi$ be an irreducible smooth representation of
$GL(n,F)$. The sum  of the semi simplifications of the  Jacquet modules
 with respect to the standard  parabolic subgroups which have Levi subgroups $GL(k,F)\times
GL(n-k,F)$, $0\leq k \leq n$, defines an element of
$R\otimes R$  (see \cite{Z} for more details). The
Jacquet modules that we consider in this paper are normalized.
We  extend this  mapping additively  to the whole $R$, and denote the extension by
$$
m^*:R  \to R\otimes R.
$$
 In this way, $R$ becomes a graded
Hopf algebra.

For  an irreducible representation  $\pi$  of $GL(n,F),$ there exist
$\rho_1,\dots,\rho_k\in\mathcal C$ 
such that $\pi$ is isomorphic to  a subquotient of $\rho_1\times\dots
\times\rho_k$. The multiset of equivalence classes $(\rho_1,\dots,\rho_k)$
is called the cuspidal support of $\pi$.

Denote by $M(D)$ the set of all finite multisets in $D$. We add multi sets in a natural way:
$$
(\delta_1,
\delta_2,\dots,\delta_k)
+
(\delta_1',
\delta_2',\dots,\delta_{k'}')=
(\delta_1,
\delta_2,\dots,\delta_k,\delta_1',
\delta_2',\dots,\delta_{k'}').
$$
For $d=(\delta_1,
\delta_2,\dots,\delta_k)\in M(D)$ take a permutation $p$ of $\{1,\dots,k\}$ such that
$$
e(\delta_{p(1)})\geq e(\delta_{p(2)})\geq\dots \geq e(\delta_{p(k)}).
$$
Then the representation
$$
\lambda(d):=\delta_{p(1)}\times
\delta_{p(2)}\times\dots\times\delta_{p(k)}
$$
(called the standard module) has a unique irreducible quotient, 
which   is denoted by
$$
L(d).
$$
The mapping $d\mapsto L(d)$ defines a bijection between $M(D)$ and the set of all equivalence classes of irreducible smooth
representations of all the general linear groups over $F$. This is a formulation of  the  Langlands classification for general linear
groups.  We can describe $L(d)$ as a unique irreducible subrepresentation of
$$
\delta_{p(k)}\times
\delta_{p(k-1)}\times\dots\times\delta_{p(1)}.
$$
The formula for the contragredient is
$$
L(\delta_1,
\delta_2,\dots,\delta_k)\tilde{\ }
\cong L(\tilde\delta_1,
\tilde\delta_2,\dots,\tilde\delta_k).
$$

A  segment in $\mathcal  C$ is a set of the form
$$
[\rho, \nu^k\rho]=\{\rho,\nu\rho,\dots, \nu^k\rho\},
$$
where $\rho\in \mathcal  C, k\in \mathbb Z_{\geq 0}$. We shall denote a segment $[\nu^{k'}\rho, \nu^{k''}\rho]$
also  by 
$$
[k',k'']^{(\rho)},
$$
 or simply by $[k',k'']$ when we fix $\rho$ (or it is clear from the context which $\rho$ is in question).
We denote $[k,k]^{(\rho)}$ simply by $[k]^{(\rho)}$.

The set of all  such  segments is denoted by
$$
\mathcal  S.
$$
For a segment $\Delta=[\rho, \nu^k\rho]=\{\rho,\nu\rho,\dots, \nu^k\rho\}\in \mathcal  S$, the
representation
$$
\nu^k\rho\times\nu^{k-1}\rho\times \dots\times \nu\rho\times\rho
$$
contains a unique irreducible subrepresentation, which is denoted by 
$$
\delta(\Delta)
$$
and a unique irreducible quotient, which is denoted by 
$$
\mathfrak s(\Delta).
$$
The representation $\delta(\Delta)$ is an essentially square  integrable representation modulo center. In this way we
get a bijection between $\mathcal  S$ and $D$. Further, $\mathfrak s(\Delta)=L(\rho,\nu\rho,\dots, \nu^k\rho)$ and
\begin{equation}
\label{m^*-seg}
m^*(\delta([\rho,\nu^k\rho]))=\sum_{i=-1}^k \delta([\nu^{i+1}\rho,\nu^k\rho]) \otimes
\delta([\rho,\nu^{i}\rho]),
\end{equation}
$$
m^*(\mathfrak s([\rho,\nu^k\rho]))=\sum_{i=-1}^k \mathfrak s
([\rho,\nu^{i}\rho])
 \otimes
\mathfrak s
 ([\nu^{i+1}\rho,\nu^k\rho]). 
$$
 Using the above bijection between $D$ and $\mathcal S$, we can express Langlands classification in terms of finite multisets $M(\mathcal S)$ in $\mathcal S$:
 $$
 L(\Delta_1,\dots,\Delta_k):=
 L(\delta(\Delta_1),\dots,\delta(\Delta_k)).
 $$
 
 The Zelevinsky classification tells that
$$
\mathfrak s(\Delta_{p(1)})\times
\mathfrak s(\Delta_{p(2)})\times\dots\times
\mathfrak s(\Delta_{p(k)}),
$$
has a unique irreducible subrepresentation, which is
  denoted by
$$
Z(\Delta_1,\dots,\Delta_k)
$$
($p$ is as above).

Since the ring $R$ is a polynomial ring over $D$,  the ring
 homomorphism $\pi \mapsto  \pi^t$ on $R$ determined by the requirement that
 $\delta(\Delta)\mapsto \mathfrak s(\Delta)$, $\Delta\in \mathcal S$, is uniquely determined by this
 condition. It is an involution, and  is called
the Zelevinsky involution. It is a special case of an 
 involution which exists for any connected reductive group, called the
  Aubert-Schneider-Stuhler  involution. This extension we shall also denote by $\pi\mapsto \pi^t$. A very important property of the Zelevinsky involution, as
 well as of the  Aubert-Schneider-Stuhler  involution, is that it carries irreducible
 representations to the irreducible ones
 (\cite{Au}, Corollaire 3.9;
 also \cite{ScSt}).

The Zelevinsky involution $^t$ on the irreducible representations  can be introduced by the requirement 
$$
L(a)^t=Z(a),
$$
for any multisegment $a$. Then we define $^t$ on the multisegments by the requirement 
$$
Z(a)^t=Z(a^t).
$$

 For $\Delta=[\rho, \nu^k\rho]\in\mathcal S$, let
 $$
 \Delta^-=[\rho, \nu^{k-1}\rho],
  $$
  and for $d=(\Delta_1,\dots,\Delta_k)\in M(\mathcal S)$ denote
  $$
  d^-=(\Delta_1^-,\dots,\Delta_k^-).
 $$
 Then the ring homomorphism $\mathcal D:R \to R$ determined by the
 requirement that $\mathfrak s(\Delta)$ goes to $\mathfrak s(\D)+\mathfrak s(\Delta^-)$ for all $\Delta\in \mathcal S$, is
 called the derivative. This is a positive mapping.
 Let $\pi\in R$ and $\mathcal D(\pi)=\sum \mathcal D(\pi)_n$, where $\mathcal D(\pi)_n$ is in the $n$-th grading group of $R$. If $k$ is the lowest index  such that $\mathcal D(\pi)_k\neq 0$, then $\mathcal D(\pi)_k$ is called the highest derivative of $\pi$, and denoted by $\text{h.d.}(\pi)$.  Obviously, the highest derivative is multiplicative (since $R$ is an integral domain). Further
 $$
 \text{h.d.}(Z(\Delta_1,\dots,\Delta_k))=Z(\Delta_1^-,\dots,\Delta_k^-)
 $$
 (see \cite{Z}).

 We now very briefly recall  basic notation for the 
classical $p$-adic groups. We   follow   
\cite{Moe-T}. 
 Fix a Witt tower $V\in \mathcal  V$  of symplectic of orthogonal vector spaces starting with an anisotropic space $V_0$ of the same type (see sections III.1 and III.2 of \cite{Ku} for details).
Consider the group of isometries of $V\in \mathcal  V$, while in the case of odd-orthogonal groups one requires additionally that the determinants are 1. The group of split rank $n$ will be denoted
 by $S_n$ 
 (for some other purposes a different indexing may be more convenient). For $0\leq k\leq n$, one chooses  a parabolic subgroup whose Levi factor is isomorphic to $GL(k,F)\t S_{n-k}$ (see \cite{Ku}, III.2)\footnote{One can find in \cite{T-Str}  matrix realizations of the symplectic and split odd-orthogonal groups. In a similar way one can  make matrix realizations also for other orthogonal groups (and for unitary groups which are discussed a little bit later).}. Then using parabolic induction one defines in a natural way multiplication
 $$
 \pi\r\s
 $$
 of a representations $\pi$ and $\s$ of $GL(k,F)$  and $S_{n-k}$ respectively.

We do not follow the case   of split even orthogonal groups in this paper, although we expect that the results of this paper hold also in this case, with  the same proofs (split even orthogonal groups are not connected, which requires some additional checkings).

Let $F'$ be a quadratic extension of $F$, and denote by $\Theta$ the non-trivial element of the Galois group. In analogous way one defines the Witt tower of unitary spaces over $F'$, starting with an anisotropic hermitian space $V_0$, and consider the isometry groups. One denotes by $S_n$ the group of $F$-split rank $n$. Here multiplication $\r$ is defined among representations of groups $GL(k,F')$ and $S_{n-k}$. Except in the last section, the classical groups that we consider in this paper are symplectic and orthogonal groups (introduced above), excluding split even orthogonal groups (what we have already mentioned). In the last section is commented the case of  the unitary groups.

An irreducible representation of a classical group will be called weakly real if it is a subquotient of a representation of the form
$$
\nu^{r_1}\rho_1\t\dots \t\nu^{r_k}\rho_k\r\s,
$$
where  
 $\rho_i\in \mathcal C$ are selfcontragredient,  
 $r_i\in \R$ and $\s$  is an irreducible cuspidal representation of a classical group.

The following theorems  reduce the unitarizability problem for classical $p$-adic groups to the  weakly real case (see \cite{T-CJM}).

\begin{theorem}
\label{unitary-red-1}
 If $\pi $ is an irreducible unitarizable representation of some $S_q$, then there exist
an irreducible unitarizable representation $\theta$ of a general linear group and a weakly real irreducible
unitarizable representation $\pi'$ of some $S_{q'}$ such that
$$
\pi\cong\theta\r\pi'.
$$
\end{theorem}

Denote by $\mathcal C_u$ the set of all unitarizable classes in $\mathcal C$. For a set $X$ of equivalence classes of irreducible representations, we denote by $\tilde X:=\{\tilde\tau;\tau\in X\}$ (recall that $\tilde\tau$ denotes the contragredient of $\tau$). Theorem 4.2 of \cite{T-CJM} gives a more precise formulation of  the above reduction:

\begin{theorem} 
\label{unitary-red-2}
Let $\mathcal C'_u$ be a subset of $ \mathcal C_u$ 
satisfying $\mathcal C'_u\cap
\widetilde
{\mathcal C_u'}=\emptyset$, such that $\mathcal C'_u\cup
\widetilde
{\mathcal C_u'}$
 contains all
$\rho \in \mathcal C_u$ which are not self contragredient.
 Denote
$$
\mathcal C'=\{\nu^\a\rho;\ \a \in \mathbb R,\ \rho\in\mathcal C'_u\}.
$$
Let $\pi $ be an irreducible unitarizable representation of some $S_q$
Then there exists an irreducible representation $\theta$ of a general linear
group with support contained in $\mathcal C'$, and a weakly real irreducible
representation $\pi'$ of some $S_{n'}$ such that
$$
\pi\cong \theta\r \pi'.
$$
Moreover, $\pi$ determines such $\theta$ and $\pi'$ up to an equivalence.
Further,
$\pi$ is unitarizable (resp. Hermitian) if and only both $\theta$ and $\pi'$
are unitarizable (resp. Hermitian).
\end{theorem}

The direct sum of Grothendieck groups $\mathcal R(S_n)$, $n\geq 0$,  is denoted by $R(S)$. 
As in the case of general linear groups, one lifts  $\rtimes$ to a mapping $R\times R(S)\to R(S)$ (again denoted
by $\rtimes$). Factorization through $R\otimes R(S)$ is denoted by $\mu$. In this way
$R(S)$ becomes an $R$-module.

We denote by
$$
s_{(k)}(\pi)
$$
 the Jacquet module of a
representation $\pi$ of $S_n$ with respect to the parabolic subgroup $P_{(k)}$.  If there exists $0\leq k\leq n$  and an irreducible cuspidal representation $\s$ of $S_{q}$, $q\leq n$, such that any irreducible sub quotient $\tau$ of $s_{(k)}(\pi)$
is isomorphic to $\theta_\tau\o\s$ for some representation $\theta_\tau$ of a general linear group, then we shall denote $s_{(k)}(\pi)$ also by
$$
s_{GL}(\pi).
$$
Then $\s$ is called a partial cuspidal support of $\pi$.

For an irreducible  representation $\pi$ of $S_n$,
the sum of the semi simplifications of $s_{(k)}(\pi)$, $0\leq k\leq n$, is
denoted by 
$$
\mu^*(\pi)
\in R\otimes R(S).
$$ 
We extend $\mu^*$ additively to 
$
\mu^*:R(S)\to R\otimes R(S).
$
With this comultiplication, $R(S)$ becomes an $R$-comodule.

Further, 
$R\otimes R(S)$ is an
$R\otimes R$-module in a natural way (the multiplication is denoted   by $\rtimes$).
Let
$\sim:R \to R$ be the contragredient map and 
$\kappa:R\otimes R\to R\otimes R$, $\sum x_i\otimes y_i \mapsto
\sum_i y_i\otimes x_i$. Denote
$$
M^*=(m \otimes \text{id}_R) \circ ( \sim \otimes m^\ast)  \circ \kappa  \circ
m^\ast.
$$
Then (\cite{T-Str} and \cite{Moe-T})
\begin{equation}
  \label{eq: structure}
\mu^*(\pi\rtimes \sigma)=M^*(\pi)\rtimes \mu^*(\sigma)
\end{equation}
for $\pi \in R$ and $\sigma \in R(S)$ (or for admissible representations $\pi$ and $\sigma$ of $GL(n,F)$ and $S_m$
respectively).
A direct consequence of the  formulas \eqref{eq: structure} and \eqref{m^*-seg}  is the following formula:
 \begin{multline*}
M^*(\delta([\nu^{a} \rho,\nu^{c}\rho]) )=\sum_{s= a-1}^{ c}\sum_{t=i}^{ c}\delta([\nu^{-s}\tilde\rho,\nu^{-a}\tilde\rho])\times
\delta([\nu^{t+1} \rho,\nu^{c}\rho]) \otimes
\delta([\nu^{s+1} \rho,\nu^{t}\rho]).
\end{multline*}

Let $\pi$ be a representation of some $GL(m,F)$. 
Then the sum of the
 irreducible subquotients of the form $\ast \otimes 1$ in $ M^{\ast}(\pi)$
 will be denoted by 
 $$
 M_{GL}^*(\pi)\o1.
 $$
   Let    
$
m^{\ast}(\pi)=\sum x\otimes
y.
$
Then easily follows that
 \begin{equation}
 \label{M*GL} 
 M_{GL}^*(\pi)=
 \sum x\times \tilde{y}.
 \end{equation}
 Let $\pi$ be a sub quotient of $\rho_1\t\dots\t\rho_l$ where $\rho_i$ are irreducible cuspidal representations of general linear groups, and let $\s$ be an irreducible cuspidal representations of $S_q$. Then the 
 $$
 \text{s.s.}(s_{GL}(\pi\r\s))=M^*_{GL}(\pi)\o\s.
 $$

 Further,    the sum of the irreducible subquotients of the form  $1\otimes \ast$ in $M^{\ast}(\tau)$
is 
\begin{equation}
1\otimes \tau.
\end{equation}

 Now we shall recall of the Langlands classification for 
 groups $S_n$ (\cite{Si2}, \cite{BW}, \cite{Ko}, \cite{Re} and \cite{Z}).
Set
$$
D_+=\{\delta\in D; e(\delta)>0\}.
$$
Let $T$ be the set of all equivalence classes of irreducible  tempered representations of $S_n$, for all $n\geq 0$.
For 
$
t=((\delta_1,\delta_2,\dots,\delta_k),\tau)\in
M(D_+)\times T
$
take a permutation $p$ of $\{1,\dots,k\}$ such that 
\begin{equation}
\label{dec}
\delta_{p(1)}\geq\delta_{p(2)}\geq\dots\geq\delta_{p(k)}.
\end{equation}
Then the representation 
$$
\lambda(t):=\delta_{p(1)}\times\delta_{p(2)}\times\dots\times\delta_{p(k)}\rtimes \tau
$$
has a unique irreducible quotient, which is denoted by
$$
L(t).
$$
The mapping
$$
t\mapsto L(t)
$$
defines a bijection from the set
$M(D_+)\times T$ onto the set of all equivalence classes of the  irreducible smooth
representations of all $S_n$, $n\geq 0$. This is the Langlands classification for classical groups. The multiplicity of $L(t)$ in $\lambda(t) $ is one.

Let $t=((\delta_1,\delta_2,\dots,\delta_k),\tau) \in M(D_+)\times T$ and suppose that a permutation $p$ satisfies \eqref{dec}. Let $\d_{p(i)}$ be a representation of $GL(n_i,F)$ and $L(t) $  a representation of $S_n$. Denote by
$$
e_*(t)=(\underbrace{\d_{p(1)},\dots,\d_{p(1)}}_{n_1\text{ times}},\dots, \underbrace{\d_{p(k)},\dots,\d_{p(k)}}_{n_k\text{ times}},\underbrace{0,\dots,0}_{n'\text{ times}}),
$$
where $n'=n-n_1-\dots-n_k$. Consider a partial ordering on $\R^n$ given by $(x_1,\dots,x_n)\leq (y_1,\dots,y_n)$ if and only if
$$
\textstyle
\sum_{i=1}^j x_i\leq \sum_{i=1}^j y_i, \quad 1\leq j\leq n.
$$
Suppose $t,t'\in M(D_+)\times T$ and $L(t')$ is a sub quotient of $\lambda(t)$. Then
\begin{equation}
\label{BPLC}
\e_*(t')\leq e_*(t), \text{ and the  equality holds in the previous relation }\iff t'=t
\end{equation}
(see section 6. of \cite{T-Comp} for the symplectic groups - this holds in the same form for the other classical groups different from the split even orthogonal groups).

For $\Delta\in \mathcal S$ define $\mathfrak c(\Delta)$ to be $e(\delta(\Delta))$.  Let
$$
\mathcal S_+=\{\Delta\in\mathcal S; \mathfrak c (\Delta)>0\}.
$$
In this way  we can define in a natural way the Langlands classification $(a,\tau)\mapsto L(a;\tau)$ using $M(\mathcal S_+)\times T$ for the parameters. 

Let $\tau$ and $\omega$ be  irreducible
representations of
$GL(p,F)$ and $S_q$, respectively, and let $\pi$ an admissible
representation of $S_{p+q}$. Then a special case of the Frobenius reciprocity tells us
$$
\text{Hom}_{_{S_{p+q}}}(\pi,\tau \rtimes  \omega)\cong
\text{Hom}_{_{GL(p,F)\times S_q}}(s_{(p)}(\pi),\tau \otimes \omega),
$$
while the second second adjointness implies
$$
\text{Hom}_{_{S_{p+q}}}(\tau \rtimes  \omega, \pi)\cong
\text{Hom}_{_{GL(p,F)\times S_q}}(\tilde\tau \otimes \omega,s_{(p)}(\pi)).
$$
We could write down
the above formulas   for the 
parabolic subgroups which are not necessarily maximal.

\section{On the irreducible sub quotients of 
\texorpdfstring{
$\nu^{\alpha + n} \rho \times 
\cdots \times
\nu^{\alpha
+ 1} \rho \times \nu^\alpha \rho \rtimes \sigma$ 
}{Lg}
}
\label{SQ}
\setcounter{equation}{0}

Let $\rho$ and $\sigma$ be  irreducible    cuspidal representations
of $GL(p,F)$ and  $S_q$ respectively, such that $\rho$ is self contragredient. Then $\rho$ is unitarizable (cuspidality implies that $\s$ is unitarizable since the center of $S_q$ is compact - more precisely, finite). 
Then
$$
\nu^{\alpha_{\rho,\sigma}}\rho\r\s
$$
reduces for some $   \alpha_{\rho,\sigma}\geq0$. This reducibility point $   \alpha_{\rho,\sigma}$ is unique 
by \cite{Si2}. 
In this paper we shall assume that
\begin{equation}
\label{1/2Z+}
\alpha_{\rho,\sigma}\in (1/2)\mathbb Z.
\end{equation}
Actually, from the recent work of J. Arthur, C. M\oe glin and J.-L. Waldspurger, this assumption is known to hold if char$(F)=0$ (Theorem 3.1.1 of \cite{Moe-paq-stab} tells this for the  quasi split case, while  \cite{MoeR}  extends it to the  non-quasi split classical groups). 

In the most of this paper we shall deal with the case
\begin{equation}
\label{>0}
\alpha_{\rho,\sigma}>0,
\end{equation}
at least in the following several sections.
We shall denote the reducibility point $   \alpha_{\rho,\sigma}$  simply by
$$
\a.
$$
So $\a>0 $ and $\a\in\frac12\Z$.   
We shall deal with irreducible sub quotients of 
$$
\nu^{\alpha + n} \rho \times \nu^{\alpha + n-1} \rho \times \cdots \times
\nu^{\alpha
+ 1} \rho \times \nu^\alpha \rho \rtimes \sigma.
$$
The above representation has a unique irreducible subrepresentation, which is denoted by
$$
\delta([\nu^\alpha \rho, \nu^{\alpha + n} \rho]; \sigma) \ \ (n\geq 0).
$$
This subrepresentation is square integrable and it is called a generalized Steinberg representation.  We have 
$$
\mu^*\left(\delta([\nu^\alpha \rho, \nu^{\alpha + n}\rho];\sigma)\right) =
\sum^n_{k=-1} \delta([\nu^{\alpha + k + 1}\rho, \nu^{\alpha +
n} \rho]) \otimes \delta([\nu^\alpha \rho, \nu^{\alpha + k}
\rho]; \sigma),
$$
$$
\delta([\nu^\alpha \rho, \nu^{\alpha + n} \rho];\sigma)\tilde{\ }
\cong
 \delta([\nu^\alpha \rho,\nu^{\alpha + n} \rho];  \tilde{\sigma}).
 $$
 Further applying the  Aubert-Schneider-Stuhler  involution, we get
 \begin{multline*}
\mu^*\left(L(\nu^{\alpha+n} \rho. \dots , \nu^{\alpha + 1}\rho, \nu^{\alpha }
\rho; \sigma)
\right) 
=
\\
\sum^n_{k=-1} 
L(\nu^{-(\alpha + n)}\rho,\dots , \nu^{-(\alpha + k+2)}\rho,\nu^{-(\alpha +
k+1)} \rho)
 \otimes
 L(\nu^{\alpha+k} \rho. \dots , \nu^{\alpha + 1}\rho, \nu^{\alpha }
\rho; \sigma).
\end{multline*}

We say that a sequence of segments $\Delta_1,\dots,\Delta_l$  is decreasing  if $\mathfrak c(\Delta_1)\geq \dots \geq \mathfrak c(\Delta_l)$.

Now we recall of Lemma 3.1 from \cite{HTd} which we shall use several times in this paper:

\begin{lemma} 
\label{comp-series}
Let $n\geq  1$. Fix an integer $c'$ satisfying $0\leq  c' \leq  n-1$. Let
$\Delta_1,\ldots,\ \Delta_k$ be a sequence of decreasing mutually disjoint  non-empty segments such
that 
\[
 \Delta_1\cup \ldots \cup \ \Delta_k=\{\nu^{\alpha+c'+1}\rho,\ldots, \nu^{\alpha+n-1}\rho,\nu^{\alpha+n}\rho\}.
 \]
Let $\Delta_{k+1},\ldots,\Delta_l$, $k<l$, be a
sequence of decreasing mutually disjoint segments satisfying
\[
\Delta_{k+1}\cup\cdots
\cup\ \Delta_l=\{\nu^{\alpha}\rho,\nu^{\alpha+1}\rho,\ldots,\nu^{\alpha+c'}\rho\},\]
 such  that
$\Delta_{k+1},\ldots,\Delta_{l-1}$ are non-empty.
Let 
$$
a=(\Delta_1,\ldots,\Delta_{k-1}),
$$
$$
b=(\Delta_{k+2},\ldots,\Delta_{l-1}).
$$
Then in $R(S)$ we have:

\begin{enumerate}
\item If $k+1<l$, then
\begin{multline*}
\hskip15mm L(a+(\Delta_k))\rtimes  L((\Delta_{k+1})+b;\delta(\Delta_l;\sigma)) =
\\
L(a+(\Delta_k,\Delta_{k+1})+b;\delta(\Delta_l;\sigma))+
L(a+(\Delta_k\cup\Delta_{k+1})+b;\delta(\Delta_l;\sigma)).
\end{multline*}

\item
If  $k+1=l$, then
\begin{equation*}
\label{eq-basic2}
L(a+(\Delta_k)) \rtimes \delta(\Delta_{k+1};\sigma) 
=L(a+(\Delta_k);\delta(\Delta_{k+1};\sigma))
+L(a;\delta(\Delta_k\cup
\Delta_{k+1};\sigma)).    \qed
\end{equation*}
\end{enumerate}
\end{lemma}

 We assume
 $$
 n\geq 1,
 $$ 
 and consider 
irreducible  subquotients of
$\nu^{\alpha+n}\rho\times\nu^{\alpha+n-1}\rho\times\cdots\times\nu^{\alpha}\rho\rtimes \sigma$.
Each irreducible subquotient can be written as
\[
\gamma=L(\Delta_1,\ldots,\Delta_k;\delta(\Delta_{k+1};\sigma))\]
for some $k\geq 0$, where
$\Delta_1,\ldots,\Delta_{k+1}$ is a sequence of decreasing mutually disjoint segments  such
that 
$$
\Delta_1\cup\ldots \cup \Delta_k\cup \Delta_{k+1}=\{\nu^{\alpha}\rho,\dots,\nu^{\alpha+n}\rho\},
$$
 and  that
$\Delta_1,\ldots,\Delta_k$ are non-empty\footnote{It is easy to see that Langlands parameter of $\g$ must be of above form. Namely, for the beginning, the tempered piece of the Langlands parameter must be square integrable (this follows from the fact that  $\rho$ is self contragredient  and the fact that $\nu^{\alpha + n} \rho \times  
\cdots 
\times \nu^\alpha \rho \rtimes \sigma$ is a regular representations, i.e. all the Jacquet modules of it are multiplicity one representations). Further, one directly sees that this square integrable representation must be some $\delta(\Delta_{k+1};\sigma)$. Now considering the support, and using the fact that $\mathfrak c(\D_i)>0$, we get that the Langlands parameter of $\g$ must be of the above form.}.

\begin{remark} Observe that 
$$
\big(\nu^{\alpha+n}\rho\t \nu^{\alpha+n-1}\rho\t \dots \t \nu^{\alpha+1}\rho \r\d(\nu^{\alpha}\rho;\s)\big)^t=
\nu^{\alpha+n}\rho\t \nu^{\alpha+n-1}\rho\t \dots \t \nu^{\alpha+1}\rho \r L(\nu^{\alpha}\rho;\s).
$$
Irreducible subquotients of $\nu^{\alpha+n}\rho\t \nu^{\alpha+n-1}\rho\t \dots \t \nu^{\alpha+1}\rho \r\d(\nu^{\alpha}\rho;\s)$ satisfy  $\D_{k+1}\ne\emptyset$, while  irreducible sub quotients  of $\nu^{\alpha+n}\rho\t \nu^{\alpha+n-1}\rho\t \dots \t \nu^{\alpha+1}\rho \r L(\nu^{\alpha}\rho;\s)$ satisfy $\D_{k+1}=\emptyset$.
From this directly follows that the  Aubert-Schneider-Stuhler  involution is a bijection between the irreducible sub quotients for which $\D_{k+1}\ne\emptyset$ and the irreducible subquotients  for which $\D_{k+1}=\emptyset$.

\end{remark}

\section{Key proposition}
\label{key}
\setcounter{equation}{0}

A bigger part  of this paper we shall spend to prove the following

\begin{proposition}
\label{prop-main} 
Let $\rho$ and $\sigma$ be  irreducible    cuspidal representations
of $GL(p,F)$ and  $S_q$ respectively, such that $\rho$ is self contragredient and that $\nu^\a\rho\r\s$ reduces for some positive  $\a\in\frac12\Z$.
 Further, let  $\gamma$ be an irreducible  subquotient   of
$\nu^{\alpha+n}\rho\times\nu^{\alpha+n-1}\rho\times\cdots\times\nu^{\alpha}\rho\rtimes \sigma$,  different from
$$
\text{
$L(\nu^{\alpha}\rho,\nu^{\alpha+1}\rho,\ldots,\nu^{\alpha+n}\rho;\sigma)$ and
$\delta([\nu^{\alpha}\rho,\nu^{\alpha+n}\rho];\sigma).
$}
$$
Then there exits an irreducible selfcontragredient unitarizable representation $\pi$ of a general linear group with support in $[-\a-n,\a+n]^{(\rho)}$,  such that the length of 
$$
\pi\r\gamma
$$
is at least 5, and that
$$
5\cdot \pi\o\gamma\not\leq \mu^*(\pi\r\gamma).
$$

\end{proposition}

We shall consider  $\gamma$ as in the proposition, and write  $\gamma=L(\Delta_1,\ldots,\Delta_k;\delta(\Delta_{k+1};\sigma))$ as in the previous section (recall, $\Delta_1,\ldots,\Delta_k$ are non-empty mutually disjoint decreasing segments, and additionally $\Delta_k,\D_{k+1}$ are decreasing if $\D_{k+1}\ne\emptyset$).
Since $\gamma$ is different from $\delta([\nu^{\alpha}\rho,\nu^{\alpha+n}\rho];\sigma)$, we
have
\[
k\geq 1,
\]
and since  $\gamma$ is  different from
$L(\nu^{\alpha}\rho,\nu^{\alpha+1}\rho,\ldots,\nu^{\alpha+n}\rho;\sigma)$ we have
\begin{equation}
\label{more than 1}
\Delta_{k+1}\neq  \emptyset
\text{  or }
\text{$\Delta_{k+1}=\emptyset$ and card $(\Delta_i)>1$ for some $1\leq  i\leq k$.}
\end{equation}
We shall first study  $\gamma$ for which $\Delta_{k+1}=\emptyset$. We split our proof of the case $\D_{k+1}$ of the above proposition into two stages. Each of them is one of the following two sections.

\section{The case of  
card\texorpdfstring{$(\Delta_k)>1$ 
and $\Delta_{k+1}=\emptyset$
}{Lg}
}
\label{first stage}
\setcounter{equation}{0}

We continue with the notation introduced in the previous section. In this section we assume card$(\Delta_k)>1$ 
and $\Delta_{k+1}=\emptyset$.
Denote 
$$
\Delta_k=[\nu^{\alpha}\rho,\nu^c\rho],
$$
$$
\Delta_u=[\nu^{-\alpha}\rho,\nu^\a\rho],
$$
$$
\Delta=\D_k\cup\D_u=[\nu^{-\alpha}\rho,\nu^c\rho].
$$
 Then
\[\alpha <  c.\]
Denote
$$
a=(\Delta_1,\Delta_2,\ldots,\Delta_{k-1}),
$$
$$
a_1=(\Delta_1,\Delta_2,\ldots,\Delta_{k-2}),\quad \text{if $a\ne\emptyset$}.
$$

For 
$
L(a,\Delta_k)\rtimes \sigma 
$
 in the Grothendieck group we have
\begin{equation}
\label{eq:case0}
L(a,\Delta_k)\rtimes\sigma
=
L(a+(\Delta_k);\sigma)
+L(a;\delta(\Delta_{k};\sigma)).
\end{equation}
We shall denote $L(a+(\Delta_k);\sigma)$ below simply by $L(a,\Delta_k;\sigma)$.

Our first goal in this sections  is to prove:

\begin{lemma}
\label{first-le}
The representation $\d(\D_u)\r L(a,\Delta_k;\sigma)$ is of length at least 5 if 
card$(\D_k)>1$.
\end{lemma}

We know from Theorem 13.2 of  \cite{T-irr} that $\d(\D_u)\r\s$ reduces. Frobenius reciprocity implies that it reduces into two non-equivalent tempered irreducible pieces. Denote them by $\tau((\Delta_u)_+;\sigma)$ and $\tau((\Delta_u)_-;\sigma).$
Now Proposition 5.3 of \cite{T-CJM} implies
\begin{equation}
\label{first and second-sq}
 L(a,\Delta_k;\tau((\Delta_u)_+;\sigma)),\quad  L(a,\Delta_k;\tau((\Delta_u)_-;\sigma)) \quad \leq \quad \delta(\Delta_u)\r L(a,\Delta_k;\sigma).
\end{equation}
Therefore, $\delta(\Delta_u)\t L(a,\Delta_k;\sigma)$ has length at least two.

 Now we shall recall of a simple Lemma 
4.2 from \cite{HTd}:

\begin{lemma}
\label{lem:help0}  If $|\D_k|>1$, then  we have
$$
L(a+(\D)) \times \nu^{\alpha}\rho
\leq
\delta(\Delta_u)\times L(a+(\Delta_k)),
$$
and the representation on the left hand side is irreducible. \qed
\end{lemma}

The above lemma now implies
$$
 L(a,\Delta) \times \nu^{\alpha}\rho \r\s
\leq
\delta(\Delta_u)\times L(a,\Delta_k)\r\s=
$$
$$
\delta(\Delta_u)\r L(a,\Delta_k;\sigma)
+\delta(\Delta_u)\r L(a;\delta(\Delta_{k};\sigma)).
$$
Since by Proposition 4.2 of \cite{T-CJM}, 
$$
L(a,\Delta, \nu^{\alpha}\rho;\s)
$$
 is a sub quotient of $L(a,\Delta) \times \nu^{\alpha}\rho \r\s$, it is also a sub quotient  of 
$$
\delta(\Delta_u)\r L(a,\Delta_k;\sigma)
+\delta(\Delta_u)\r L(a;\delta(\Delta_{k};\sigma)).
$$
Suppose
$$
 L(a,\Delta, \nu^{\alpha}\rho;\s)
 \leq
 \delta(\Delta_u)\r L(a;\delta(\Delta_{k};\sigma)).
$$
Then
$$
 L(a,\Delta, \nu^{\alpha}\rho;\s)
 \leq
 \lambda(a)\r \tau =\lambda(a,\tau),
$$
where $\tau$ is some irreducible subquotient   of $\d(\D_u)\r \delta(\Delta_{k};\sigma)$.
Now \eqref{BPLC}
 implies that this is not possible (since $\a>0$).
Therefore
\begin{equation}
\label{third-sq}
L(a,\Delta, \nu^{\alpha}\rho;\s)\leq \delta(\Delta_u)\r L(a,\Delta_k;\sigma).
\end{equation}
Now we know that $\delta(\Delta_u)\t L(a,\Delta_k;\sigma)$ has length at least three, since obviously the representation $L(a,\Delta, \nu^{\alpha}\rho;\s)$  is not in $ \{L(a,\Delta_k;\tau((\Delta_u)_+;\sigma)),  L(a,\Delta_k;\tau((\Delta_u)_-;\sigma))\}$ (consider the tempered parts of the Langlands parameters).

From Theorem in the introduction of  \cite{T-seg} (see also \cite{MaT}) we know that $\d(\D)\r\s$ has two nonequivalent irreducible sub representations, and that they are square integrable. They are denoted there by $\d(\D_+;\s)$ and $\d(\D_-;\s)$.
This and Proposition 5.3 of \cite{T-CJM} imply
$$
L(a,\nu^{\a}\rho;\d(\D_\pm;\s))\leq L(a,\nu^{\a}\rho)\r \d(\D_\pm;\s)\leq L(a)\t \nu^{\a}\rho\r \d(\D)\r\s\leq 
$$
$$
L(a)\t\d(\D_k)\t \d(\D_u)\r\s.
$$
Therefore
$$
L(a,\nu^{\a}\rho;\d(\D_\pm;\s))\leq 
L(a)\t\d(\D_k)\t \d(\D_u)\r\s.
$$
If $a=\emptyset$, then formally 
\begin{equation}
\label{want to prove}
L(a,\nu^{\a}\rho;\d(\D_\pm;\s))\leq 
L(a,\D_k)\t \d(\D_u)\r\s
\end{equation}
since then $L(a)\t\d(\D_k)=L(\D_k)=L(a,\D_k)$.

Now we shall show that \eqref{want to prove} holds also if
$a\ne\emptyset$. In this case we have
$$
L(a,\nu^{\a}\rho;\d(\D_\pm;\s))\leq 
L(a)\t\d(\D_k)\t \d(\D_u)\r\s=
$$
$$
L(a,\D_k)\t \d(\D_u)\r\s +
L(a_1,\D_{k-1}\cup\D_k)\t \d(\D_u)\r\s
$$
(here we have used that $L(a)\t\d(\D_k)=L(a,\D_k) +
L(a_1,\D_{k-1}\cup\D_k)$, which is easy to prove).
Suppose
$$
L(a,\nu^{\a}\rho;\d(\D_\pm;\s))\leq
L(a_1,\D_{k-1}\cup\D_k)\t \d(\D_u)\r\s.
$$
Observe that $L(a,\nu^{\a}\rho;\d(\D_\pm;\s))\h L(\tilde a,\nu^{-\a}\rho)\r \d(\D_\pm;\s) 
\h L(\tilde a,\nu^{-\a}\rho)\t \d(\D)\r\s.$ This implies that the Langlands quotient $L(a,\nu^{\a}\rho;\d(\D_\pm;\s))$ has in the GL-type Jacquet module an irreducible sub quotient 
$$
\b
$$
 which has exponent $c$ in its Jacquet module, but has not $c+1$.

From the case of general linear groups we know $L(a_1,\D_{k-1}\cup\D_k)\leq L(a_1)\r \d(\D_{k-1}\cup\D_k)$ (see for example Proposition A.4 of \cite{T-AENS}). Now application of  tensoring, parabolic induction and Jacquet modules imply
$$
s_{GL}(L(a_1,\D_{k-1}\cup\D_k)\t \d(\D_u)\r\s)\leq
s_{GL}(L(a_1)\t \d(\D_{k-1}\cup\D_k)\t \d(\D_u)\r\s).
$$
Therefore,
 (on the level of semisimplifications) we have
$$
s_{GL}(L(a_1,\D_{k-1}\cup\D_k)\t \d(\D_u)\r\s)\leq \hskip60mm
$$
$$
M^*_{GL}((L(a_1))\t M^*_{GL}(\d(\D_{k-1}\cup\D_k))\t M^*_{GL}(\d(\D_u))\r(1\o \s).
$$
Now we shall examine how we can can get exponents $c$ and $c+1$ in the support of the left hand side tensor factor  of the  last line. Since the left hand side is a product of three terms, we shall analyze each of them. Recall $\D_u=[\nu^{-\a}\rho,\nu^\a\rho]$. Since $0<\a<c$, now \eqref{M*GL} and \eqref{m^*-seg} imply that neither $\nu^c\rho$ nor $\nu^{c+1}\rho$ is in the support of $M^*_{GL}(\d(\D_u))$.

Further, recall that $a=(\D_1,\dots,\D_{k-1})$ and $a_1=(\D_1,\dots,\D_{k-2})$. Since $\D_1\cup\dots\cup \D_{k-1}=[\nu^{c+1}\rho,\nu^{\a+n}\rho]$, we have $\D_1\cup\dots\cup \D_{k-2}=[\nu^{t}\rho,\nu^{\a+n}\rho]$ for some $t\geq c+2$ (clearly, $c+2>0$). This implies that no one of $\nu^{\pm c}\rho$ or $\nu^{\pm (c+1)}\rho$ is in the support of $L(a_1) $ or $L(a_1)\tilde{\ }$. This implies that neither $\nu^c\rho$ nor $\nu^{c+1}\rho$ is in the support of $M^*_{GL}(L(a_1))$ (use the fact that $L(a_1)\leq \prod_{i=1}^s \tau_i$ for some $\tau_i$ in the support of $L(a_i)$, and the formula that $M^*_{GL}(\tau_i)=\tau_i+\tilde\tau_i$ since $\tau_i$ are cuspidal).

Since the exponent $c$ shows up in the support of $\b$, it must show up in 
$$
M^*_{GL}(\d(\D_{k-1}\cup\D_k))=M^*_{GL}(\d([\a,d]))=
\sum_{x= \a-1}^{ d}
\delta([-x,-\a])\times
\delta([x+1,d]) ,
$$
where $c+1\leq d$. Now the above formula  for $M^*_{GL}(\d([\a,d])) $ implies that whenever we have in the support $c$, we must have it in a segment which ends with $d$, and therefore, we must have in the support also $c+1$. Therefore,   $\b$ cannot be a sub quotient of $L(a_1,\D_{k-1}\cup\D_k)\t \d(\D_u)\r\s$. This contradiction implies
$$
L(a,\nu^{\a}\rho;\d(\D_\pm;\s))\not\leq
L(a_1,\D_{k-1}\cup\D_k)\t \d(\D_u)\r\s.
$$
Now the following relation (which we have already observed above)
$$
L(a,\nu^{\a}\rho;\d(\D_\pm;\s))\leq 
L(a,\D_k)\t \d(\D_u)\r\s +
L(a_1,\D_{k-1}\cup\D_k)\t \d(\D_u)\r\s
$$
implies
$$
L(a,\nu^{\a}\rho;\d(\D_\pm;\s))\leq 
L(a,\D_k)\t \d(\D_u)\r\s.
$$

Therefore (in both cases) we have 
$$
L(a,\nu^{\a}\rho;\d(\D_\pm;\s))\leq
L(a,\D_k)\t \d(\D_u)\r\s=
$$
$$
\d(\D_u)\r L(a,\Delta_k;\sigma)
+
\d(\D_u)\r L(a;\delta(\Delta_{k};\sigma)).
$$
Suppose
$$
L(a,\nu^{\a}\rho;\d(\D_\pm;\s))\leq
\d(\D_u)\r L(a;\delta(\Delta_{k};\sigma)).
$$
One directly sees that in the GL-type Jacquet module of the left hand side we have an irreducible term in whose support appears exponent $-\a$ two times. 

Observe $\d(\D_u)\r L(a;\delta(\Delta_{k};\sigma)) \leq \d(\D_u)\t L(a)\r\delta(\Delta_{k};\sigma)$. For $\d(\D_u)\t L(a)\r\delta(\Delta_{k};\sigma)$, the exponent $-\a$ which cannot come neither from $M^*_{GL}(L(a))$ nor from $\mu^*(\delta(\Delta_{k};\sigma))$.
Therefore, it must come from
 $$
 M^*_{GL}(\d(\D_u))=\sum_{x= -\a-1}^{ \a}
\delta([-x,\a])\times
\delta([x+1,\a]).
 $$ 
 This implies that we can have the exponent $-\a$ at most once in the GL-part of Jacquet module of the right hand side. This contradiction implies that the
  inequality which we have supposed is false. This implies
$$
L(a,\nu^{\a}\rho;\d(\D_\pm;\s))\leq
\d(\D_u)\r L(a,\Delta_k;\sigma).
$$
Therefore, $\delta(\Delta_u)\t L(a,\Delta_k;\sigma)$ has length at least five. This completes the proof of the lemma.

The second aim of this section is to prove the following

\begin{lemma}
\label{le-J-1}
The multiplicity of 
$$
\delta(\Delta_u)\o L(a,\Delta_k;\sigma)
$$
in
$$
\mu^*(\delta(\Delta_u)\r L(a,\Delta_k;\sigma))
$$
is at most 4 if card$(\D_k)>1$.

\end{lemma} 

\begin{proof} Denote $\b:=\delta(\Delta_u)\o L(a,\Delta_k;\sigma)$. Recall
$$
M^*(\delta(\Delta_u))
=
M^*(\delta([-\a,\a]) )=\sum_{x= -\a-1}^{ \a}\sum_{y=x}^{ \a}\delta([-x,\a])\times
\delta([y+1,\a]) \otimes
\delta([x+1,y])
$$
Now if we take from $\mu^*(L(a,\Delta_k;\sigma))$ the term $1\o L(a,\Delta_k;\sigma)$, to get $\b$ for a sub quotient we need to take from $M^*(\delta(\Delta_u))$ the term $\delta(\Delta_u)\o1$, which 
shows up there two times. This gives multiplicity two of $\b$.

Now we consider  terms from $\mu^*(L(a,\Delta_k;\sigma))$ different from $1\o L(a,\Delta_k;\sigma)$ which can give $\b$ after multiplication with a term from $M^*(\delta(\Delta_u))$ (then a term from $M^*(\delta(\Delta_u))$ that can give $\b$ for a sub quotient is obviously different from 
$\d(\D_u)\o1$,
which implies that we have $\nu^{\a}\rho$ in the support of the left hand side tensor factor).
The above formula for $M^*(\delta(\Delta_u))$ and the set of possible factors of $L(a,\Delta_k;\sigma)$ (which is $\nu^{\pm\a}\rho,\nu^{\pm(\a+1)},\dots $)  imply that we need to have $\nu^{-\a}\rho$ on the left hand side of the tensor product of that term from $\mu^*(L(a,\Delta_k;\sigma))$. For such a term from $\mu^*(L(a,\Delta_k;\sigma))$,  considering the support we see that we have two possible terms from  $M^*(\delta(\Delta_u))$. They are
$
\delta([-\a+1,\a]) \otimes
[-\a]
$ 
and 
$
\delta([-\a+1,\a]) 
\otimes
[\a]
$. 
Each of them will give multiplicity at most one (use the fact that here on the left and right hand side of $\o$ we are in the regular situation).
\end{proof}

\section{The case of  
\texorpdfstring{card$(\Delta_k)=1$ 
and $\Delta_{k+1}=\emptyset$
}{Lg}
}
\label{second stage}
\setcounter{equation}{0}

We continue with the notation introduced in the  section \ref{key}. In this section we assume card$(\Delta_k)=1$ 
and $\Delta_{k+1}=\emptyset$.
As we already noted in \eqref{more than 1}, we consider the case when card($\D_i)>1$ for some $i$. Denote maximal such index by $k_0$. Clearly, 
$$
k_0<k.
$$
Write 
$$
\Delta_{k_0}=[\nu^{\alpha+k-k_0}\rho,\nu^c\rho]=[\nu^{\alpha'}\rho,\nu^c\rho], 
$$
$$
\Delta_u=[\nu^{-\alpha'}\rho,\nu^{\a'}\rho],
$$
$$
\Delta=[\nu^{-\alpha'}\rho,\nu^c\rho],
$$
$$
\D_1=[\a,\a'-1],
$$
$$
b=[\a,\a'-1]^t=([\a],[\a+1]\dots,[\a'-1])\ne \emptyset.
$$
 Then
$$
\alpha' <  c.
$$
Let
$$
a=(\Delta_1,\Delta_2,\ldots,\Delta_{k_0-1}),
$$
$$
a_1=(\Delta_1,\Delta_2,\ldots,\Delta_{k_0-2}),\quad \text{if $a\ne\emptyset$}.
$$
Then
$$
(\D_1,\dots,\D_k)=(a,\D_{k_0},b).
$$

We shall study 
$
L(a,\Delta_{k_0},b)\rtimes \sigma .
$
The previous lemma implies that in the Grothendieck group we have
\begin{equation}
\label{eq:case1}
L(a,\Delta_{k_0},b)\rtimes\sigma
=
L(a,\Delta_{k_0},b;\sigma)
+L(a;\Delta_{k_0},\nu^{\a'-1},\dots,\nu^{\a+1};\d(\nu^{\a};\sigma)).
\end{equation}
Our first goal in this section is to prove the following

\begin{lemma}
\label{sec-le}
Then length of the representation $\d(\D_u)\r L(a,\D_{k_0},b;\sigma)$ is at least 5 if  $k_0<k$ and card$(\D_{k_0})>1$.
\end{lemma}

First we get that we have two non-equivalent sub quotients
\begin{equation}
L(a,\Delta_{k_0},b;\tau((\D_u)_\pm;\s))\leq \delta(\Delta_u)\times L(a,\Delta_{k_0},b;\s)
\end{equation}
in the same way as in the previous section. Therefore, the length is at least two.

Now we shall prove  the following simple
\begin{lemma}
\label{lem:help1}  If $|\D_k|=1$, then  we have
$$
L(a,\Delta_{k_0}\cup \D_u,b, 
\nu^{\alpha'}\rho)
\leq
\delta(\Delta_u)\times L(a,\Delta_{k_0},b).
$$
\end{lemma}

\begin{proof}
Since in general $\ L(\Delta_1',\Delta_2',\dots,\Delta'_m)^t=
Z(\Delta_1',\Delta_2',\dots,\Delta'_m)$,  
it is enough to prove the lemma for the Zelevinsky classification.

  The highest (non-trivial) derivative of
$\mathfrak s(\Delta_u)\times Z(a,\Delta_{k_0},b)$ is $\mathfrak s(\Delta_u^-)\times
Z(a^-,\Delta_{k_0}^-)$. One can easily see that one 
subquotient of the last representation  is
$ Z(a^-,\Delta^-).$ Therefore, there must exist an irreducible subquotient of $\mathfrak s(\Delta_u)\times Z(a,\Delta_{k_0},b)$ whose highest derivative is $ Z(a^-,\Delta^-).$ The support and highest
derivative completely  determine the irreducible representation. One directly sees that
this representation is 
$
Z(a,\Delta_{k_0}\cup \D_u,b, \nu^{\alpha'}\rho)$. The proof is now complete
\end{proof}

The above lemma implies
$$
L(a,\Delta_{k_0}\cup \D_u,b, 
\nu^{\alpha'}\rho;\s)
\leq
L(a,\Delta_{k_0}\cup \D_u,b, 
\nu^{\alpha'}\rho)\r\s
\leq
\delta(\Delta_u)\times L(a,\Delta_{k_0},b)\r\s.
$$
By Lemma \ref{comp-series} we  have for the right hand side \quad 
$
\delta(\Delta_u)\t L(a,\Delta_{k_0},b)\rtimes\sigma
=
$
$$
\delta(\Delta_u)\t L(a,\Delta_{k_0},b;\sigma)
+\delta(\Delta_u)\t L(a;\Delta_{k_0},\nu^{\a'-1},\dots,\nu^{\a+1};\d(\nu^{\a};\sigma)).
$$
This implies \quad 
$
L(a,\Delta_{k_0}\cup \D_u,b, 
\nu^{\alpha'}\rho;\s)
\leq
$
$$
\delta(\Delta_u)\t L(a,\Delta_{k_0},b;\sigma)
+\delta(\Delta_u)\t L(a;\Delta_{k_0},\nu^{\a'-1},\dots,\nu^{\a+1};\d(\nu^{\a};\sigma)).
$$
Suppose 
$$
L(a,\Delta_{k_0}\cup \D_u,b, 
\nu^{\alpha'}\rho;\s)
\leq
\delta(\Delta_u)\t L(a;\Delta_{k_0},\nu^{\a'-1},\dots,\nu^{\a+1};\d(\nu^{\a};\sigma)).
$$
Using the properties of the irreducible sub quotients of the standard modules in the Langlands classification, we now  conclude in the same way as in the last section that this cannot be the case (the sum of all exponents on the left hand side which are not coming from the tempered representation of the classical group is the same as the sum of exponents of cuspidal representations which show up in the segments of $a$, in $\D_{k_0}$, and $\a'-1,\a'-2,\dots,\a'+1,\a'$, while the corresponding sum of the standard modules which come from the right hand side is the sum of exponents of cuspidal representations which show up in the segments of $a$, in $\D_{k_0}$, and $\a'-1,\a'-2,\dots,\a'+1$, which is strictly smaller (for $\a>0$) then we have on the left hand side).

This implies
\begin{equation}
L(a,\Delta_{k_0}\cup \D_u,b, 
\nu^{\alpha'}\rho;\s)
\leq
\delta(\Delta_u)\t L(a,\Delta_{k_0},b;\sigma).
\end{equation}
Therefore, $\delta(\Delta_u)\t L(a,\Delta_{k_0},b;\sigma)$ has length at least three.

The following  (a little bit longer) step will be to show that 
$\delta(\Delta_u)\t L(a,\Delta_{k_0},b;\sigma)$ has two additional irreducible sub quotients.

We start this step with an observe that
\begin{equation}
\label{start}
L(a,\nu^{\a'}\rho,b;\d(\D_\pm;\s))\leq
\hskip 60mm
\end{equation}
$$
\hskip 30mm L(a)\t L(b) \t \nu^{\a'}\rho\t\d(\D) \r\s \leq 
L(a)\t L(b)\t\d(\D_{k_0})\t \d(\D_u)\r\s
$$
If $a=\emptyset$, then formally 
$$
L(a,b,\nu^{\a'}\rho;\d(\D_\pm;\s))\leq
L(a,\D_{k_0})\t L(b)\t \d(\D_u)\r\s
$$
since $L(a)\t\d(\D_{k_0})=L(\D_{k_0})=L(a,\D_{k_0})$.

We shall now show that the above inequality holds also if $a\ne\emptyset$. Then staring with \eqref{start} we get
$$
L(a,b,\nu^{\a'}\rho;\d(\D_\pm;\s))\leq
L(a) \t\d(\D_{k_0}) \t L(b)\t \d(\D_u)\r\s=
$$
$$
L(a,\D_{k_0})\t L(b)\t \d(\D_u)\r\s +
L(a_1,\D_{k_0-1}\cup\D_{k_0})\t L(b)\t \d(\D_u)\r\s.
$$
Suppose
$$
L(a,b,\nu^{\a'}\rho;\d(\D_\pm;\s))
\leq
L(a_1,\D_{k_0-1}\cup\D_{k_0})\t L(b)\t \d(\D_u)\r\s.
$$

Observe that $L(a,\nu^{\a'}\rho,b;\d(\D_\pm;\s))\h L(\tilde a,\nu^{-\a'}\rho,\tilde b)\r \d(\D_\pm;\s) 
\h L(\tilde a,\nu^{-\a'}\rho,\tilde b)\r \d(\D)\t\s.$

This implies that the Langlands quotient has in the GL-type Jacquet module an irreducible sub quotient which has exponent $c$ in its Jacquet module, but does not have  $c+1$.

Observe that (on the level of semisimplifications) we have 
$$
s_{GL}(L(a_1,\D_{k_0-1}\cup\D_{k_0})\t L(b)\t \d(\D_u)\r\s)\leq
s_{GL}(L(a_1)\t \d(\D_{k_0-1}\cup\D_{k_0})\t L(b)\t \d(\D_u)\r\s)=
$$
$$
M^*_{GL}(L(a_1))\t M^*_{GL}(\d(\D_{k_0-1}\cup\D_{k_0}))\t M^*_{GL}(L(b))\t M^*_{GL}(\d(\D_u))\r(1\o \s).
$$

We cannot get any one of exponents $c$ and $c+1$ from $M^*_{GL}((L(a_1))$ or $M^*_{GL}(\d(\D_u))$ or $M^*_{GL}(L(b))$ (consider support as in the previous section). Therefore, it must come from
$$
M^*_{GL}(\d(\D_{k_0-1}\cup\D_{k_0}))=M^*_{GL}(\d([\a',d]))=\sum_{x= \a'-1}^{ d}
\delta([-x,-\a'])\times
\delta([x+1,d]),
$$
where $c+1\leq d$. The above formula  for $M^*_{GL}(\d([\a',d])) $ implies that whenever we have in the support $c$, it must come from  a segment which ends with $d$, and therefore, we must have in the support also $c+1$. Therefore, we cannot have only $c$. In this way we have proved that (in both cases)
$$
L(a,\nu^{\a'}\rho,b;\d(\D_\pm;\s))\leq
L(a,\D_{k_0})\t L(b)\t \d(\D_u)\r\s=
$$
$$
L(a,\D_{k_0},b)\t \d(\D_u)\r\s
+
L(a,[\a'-1,c])\t L([\a,\a'-2]^t)\t \d(\D_u)\r\s.
$$

Suppose
$$
L(a,\nu^{\a'}\rho,b;\d(\D_\pm;\s))\leq
L(a,[\a'-1,c])\t L([\a,\a'-2]^t)\t \d(\D_u)\r\s.
$$
Observe that
$$
L(a,\nu^{\a'}\rho,b;\d(\D_\pm;\s))\h 
\d(\tilde\D_1)\t\dots\t\d(\tilde\D_{k_0-1})\t 
\nu^{-\a'}\rho \t
\nu^{-\a'+1}\rho\t\dots\t \nu^{-\a}\rho\r \d(\D_\pm;\s)),
$$
which implies (because of unique irreducible subrepresentation of the right hand side)
$$
L(a,\nu^{\a'}\rho,b;\d(\D_\pm;\s))\h  L(\tilde a)\t \d([-\a',-\a])^t\r \d([-\a',c]_\pm;\s)
$$
$$
\h L(\tilde a)\t \d([-\a',-\a])^t\t \d([-\a',c])\r\s.
$$
Therefore, we have in the Jacquet module of the left hand side the irreducible representation
$$
 L(\tilde a)\o \d([-\a',-\a])^t\t \d([-\a',c])\o\s.
$$

Now we shall examine how we can get this from 
$$
\mu^*(L(a)\t\d([\a'-1,c])\t L([\a,\a'-2]^t)\t \d(\D_u)\r\s)
$$
a term of the form $\b\o\g$ such that the support of $\b$ is the same as of $\tilde a$. Grading and disjointness of supports "up to a contragredient" imply that we need to take $\b$ from $M^*(L(a))$ (we must take $L(\tilde a)\o1)$. This implies (using transitivity of Jacquet modules) that we need to have 
$$
\d([-\a',-\a])^t\t \d([-\a',c])\o\s\leq \mu^*(\d([\a'-1,c])\t L([\a,\a'-2]^t)\t \d(\D_u)\r\s),
$$
which implies
$$
\d([-\a',-\a])^t\t \d([-\a',c])\o\s\leq M^*_{GL}(\d([\a'-1,c])\t L([\a,\a'-2]^t)\t \d(\D_u))\o\s.
$$
Observe that in   the multisegment that represents  the left hand side, we have $[-\a']$. In particular, we have a segment which ends with $-\a'$. 

Such a segment (regarding ending at $-\a'$) we cannot get from $M^*_{GL}(L([\a,\a'-2]^t))$ (because of the support). Neither we can get it from $ M^*_{GL}(\d(\D_u))$ because of the formula:
$$
 M^*_{GL}(\d(\D_u))=\sum_{x= -\a'-1}^{ \a'}
\delta([-x,\a'])\times
\delta([x+1,\a']).
 $$ 

The only possibility is $M^*_{GL}(\d([\a'-1,c]))$. But segments coming from this term 
end with $c$ or $-\a'+1$. So we cannot get $-\a'$ for end.

Therefore, we have got a contradiction.

This implies
$$
L(a,\nu^{\a'}\rho,b;\d(\D_\pm;\s))\leq
L(a,\D_{k_0},b)\t \d(\D_u)\r\s.
$$
$$
=\d(\D_u)\r L(a,\D_{k_0},b;\sigma)
+
\d(\D_u)\r L(a,\D_{k_0},[\a+1,\a'-1]^t;\delta([\a];\sigma)).
$$
Suppose
$$
L(a,\nu^{\a'}\rho,b;\d(\D_\pm;\s))\leq
\d(\D_u)\r L(a,\D_{k_0},[\a+1,\a'-1]^t;\delta([\a];\sigma)).
$$

One directly sees (using the Frobenius reciprocity) that in the GL-type Jacquet module of the left hand side we have an irreducible term in whose support appears exponent $-\a$ two times.

This cannot happen on the right hand side. To see this, observe that the right hand side is
$$
\leq \d(\D_u)\t L(a,\D_{k_0},[\a+1,\a'-1]^t)\r \delta([\a];\sigma)).
$$
Observe that we cannot get $-\a$ from $L(a,\D_{k_0},[\a+1,\a'-1]^t)$ (consider support, and its contragredient). We cannot get it from $\delta([\a];\sigma))$ (since $\mu^*(\delta([\a];\sigma))=1\o \delta([\a];\sigma)+[\a]\o\sigma$). From  the formula for $M^*_{GL}(\d(\D_u))$) we see that we can get $-\a$ at most once (since it is negative).

 Therefore, this inequality cannot hold. This implies
\begin{equation}
L(a,\nu^{\a'}\rho,b;\d(\D_\pm;\s))\leq
\d(\D_u)\r L(a,\D_{k_0},b;\sigma).
\end{equation}
Therefore, $\d(\D_u)\r L(a,\D_{k_0},b;\sigma)$ has length at least five. The proof of the lemma is now complete.

Our second goal in this section is to prove

\begin{lemma}
\label{le-J-2}
The multiplicity of 
$$
\delta(\Delta_u)\o L(a,\D_{k_0},b;\sigma)
$$
in
$$
\mu^*(\delta(\Delta_u)\r L(a,\D_{k_0},b;\sigma))
$$
is at most 4 if  $k_0<k$ and card$(\D_{k_0})>1$.

\end{lemma} 

\begin{proof} Denote
$$
\b:=\delta(\Delta_u)\o L(a,\D_{k_0},b;\sigma).
$$
If we take from $\mu^*(L(a,\D_{k_0},b;\sigma))$ the term $1\o L(a,\D_{k_0},b;\sigma)$, to get $\b$ for a sub quotient, we need to take from $M^*(\delta(\Delta_u))$ the term $ \delta(\Delta_u)\o1$  (we can take it two times - see the above formula for $M^*(\delta(\Delta_u))$). In this way we get multiplicity two.

Now we consider in $\mu^*(L(a,\D_{k_0},b;\sigma))$  terms different from $1\o L(a,\D_{k_0},b;\sigma)$ which can give $\b$ for a sub quotient.

Observe that by Lemma \ref{comp-series}
$$
L(a,\D_{k_0},b;\sigma)\leq L(a,[\a'+1,c])\r L([\a,\a']^t;\sigma) .
$$
Now the support forces that from $M^*(L(a,[\a'+1,c]))$ we must take $1\o L(a,[\a'+1,c])$. The only possibility which would not give a term of the form $1\o -$ is to take from $M^*(L([\a,\a']^t;\sigma))$
the term $[-\a']\o L([\a,\a'-1]^t;\sigma)$ (observe that we need to get a non-degenerate representation on the left hand side of $\o$ and use the formula $\mu^*(L([\a,\a']^t;\sigma))=\sum_{i=0}^{\a'-\a+1}L([\a'-i+1,\a']^t)\tilde{\ }\o L([\a,\a'-i]^t;\sigma)$ which follows directly from the formula for $\mu^*(\d([\a,\a'];\sigma))$).

Now we need to take from $M^*(\D_u))$ a term of form $\d([-\a'+1,\a'])\o-$, for which we have two possibilities (analogously as in the proof of former corresponding lemma; use the formula for $M^*(\D_u))$). Since on the left and right hand side of $\o$ we have regular representations (which are always multiplicity one), we get in this way at most two additional multiplicities. Therefore, the total multiplicity is at most 4.
\end{proof}

\section{End of proof of Proposition \ref{prop-main}}
\label{end-proof-prop}
\setcounter{equation}{0}

We continue with the notation introduced in the section \ref{key}.

A direct consequence of the claims that we have proved in the last two sections  is the following

\begin{corollary} Let $\Delta_{k+1}\ne\emptyset$ and $k\geq 1$.
Consider 
$$
L(\D_1,\dots,\D_k;\d(\D_{k+1};\s))^t=L(\D_1',\dots,\D_{k'}';\s).
$$
Then card($\D_i')>1$ for some $i$. Denote maximal such index by $k_0'$. 
Write 
$$
\Delta_{k_0}=[\nu^{\alpha+k-k_0}\rho,\nu^c\rho]=[\nu^{\alpha'}\rho,\nu^c\rho]. 
$$
Denote
$$
\Delta_u=[\nu^{-\alpha'}\rho,\nu^{\a'}\rho].
$$
Then  

\begin{enumerate}
\item The length of $\d(\D_u)^t\r L(\D_1,\dots,\D_k;\d(\D_{k+1};\s))$ is at least 5.

\item The multiplicity of of $\d(\D_u)^t\o L(\D_1,\dots,\D_k;\d(\D_{k+1};\s))$ in  
$$
\mu^*(\d(\D_u)^t\r L(\D_1,\dots,\D_k;\d(\D_{k+1};\s)))
$$
 is at most 4.
\end{enumerate}
\end{corollary}

\begin{proof}
Denote 
$$
\tau=L(\D_1,\dots,\D_k;\d(\D_{k+1};\s)).
$$
Now by Lemmas \ref{first-le} and \ref{sec-le} we know that
$$
\d(\D_u)\r \tau^t
$$ 
is a representation of length at least 5.
This implies that 
$$
\d(\D_u)^t\r\tau
$$  
has length $\geq 5$.

Further, Lemmas \ref{le-J-1} and \ref{le-J-2} imply that the multiplicity of $\d(\D_u)\o \pi^t$ in $\mu^*(\d(\D_u)\r \pi^t)$ is at most 4. This implies that 
  the multiplicity of $\d(\tilde \D_u)^t\o\pi\cong \d( \D_u)^t\o\pi$ in $\mu^*(\d(\D_u)^t\r\pi)$ is $\leq 4$.
This completes the proof of the corollary.
\end{proof}

This corollary, together with Lemmas \ref{first-le}, \ref{le-J-1}, \ref{sec-le} and imply Proposition \ref{prop-main}.

Later  in this paper we shall show how Proposition \ref{prop-main} implies in a simple way Theorem \ref{th-int-2}. Now we give another proof of the following result of Hanzer, Jantzen and Tadi\'c: 

\begin{theorem}
\label{main}
If 
$\gamma$
is an irreducible  sub quotient of 
$$
\nu^{\alpha+n}\rho\times\nu^{\alpha+n-1}\rho\times\cdots\times\nu^{\alpha}\rho\rtimes \sigma
$$
  different from $L([\a,\a+n]^{(\rho)};\s)$ and $L([\a+n]^{(\rho)},[\a+n-1]^{(\rho)},\dots,[\a]^{(\rho)},\s)$, then 
$$
L(\D_1,\dots,\D_k;\d(\D_{k+1};\s))
$$
is not unitarizable. 
\end{theorem}

\begin{proof} Chose $\pi$ as in Proposition \ref{prop-main}.
Suppose that $\gamma$ is unitarizable. Then $\pi\r\gamma$ is unitarizable. Let $\tau$ be a sub quotient of $\pi\r\gamma$. Then $\tau\h \pi\r\gamma$. Now the Frobenius reciprocity implies that $\pi\o \gamma$ is in the Jacquet module of $\tau$.

We know that $\pi\r\gamma$ has length $\geq 5$. This (and unitarizability) implies that there are (at least) 5 different  irreducible subrepresentations of $\pi\r\gamma$. Denote them by $\tau_1,\dots,\tau_5$. Then
$$
\tau_1\oplus\dots\oplus\tau_5\h \pi\r\gamma.
$$
Since the Jacquet functor is exact, the first part of the proof implies that the multiplicity of $\pi\o\gamma$ in the Jacquet module of $\pi\r\gamma$ is at least 5. This contradicts to the second claim of Proposition \ref{prop-main}. The proof is now  complete.
\end{proof}

\section{Jantzen decomposition}
\label{Jantzen}
\setcounter{equation}{0}

In this section we shall recall of the basic results of C. Jantzen from \cite{J-supp}. We shall write them  in a slightly different way then in \cite{J-supp}. They are written there for the symplectic and the split odd-orthogonal series of groups. Since the Jantzen's paper is based on the formal properties of the representation theory of these groups (contained essentially in the structure of the twisted Hopf module which exists on the representations of these groups - see \cite{T-Str}), 
the results of \cite{J-supp} apply also whenever this structure is established. Therefore, it also holds for all the classical $p$-adic groups considered in \cite{Moe-T}\footnote{In the case of unitary groups one needs to replace usual contragredient by the contragredient twisted by the non-trivial element of the Galois group of the involved quadratic extension (see \cite{Moe-T}). The case of disconnected  even split orthogonal group is 
considered in \cite{J-CJM}.}.

A representation $\rho\in\mathcal C$\footnote{Recall, $\mathcal C$ is the set of all irreducible cuspidal representations of general linear groups.} is called a factor of an irreducible representation $\g$ of a classical group, if there exists an irreducible subquotient $\tau\o\g_{cusp}$ of $s_{GL}(\g)$ such  that $\rho$ is in the support of $\tau$.

 We have  already  used above the well known notion of (cuspidal) support of an irreducible representation of a general linear group introduced by J. Bernstein and A. V. Zelevinsky.
 Now we shall  introduce such notion for classical groups.
 We shall fix below  an irreducible cuspidal representation $\s$ of a classical group.
Let $X \subseteq \mathcal C$ and suppose that $X$ is self contragredient, i.e. that 
$$
\tilde X=X,
$$
where $\tilde X=\{\tilde\rho;\rho\in\ X\}.$
Following C. Jantzen, one says that an irreducible representation $\g$ of a classical group is supported by $X\cup\{\s\}$ 
if there exist $\rho_1,\dots,\rho_k$
from $X$ 
such that
$$
\g\leq \rho_1\t\dots\t\rho_k\r\s.
$$
For not-necessarily irreducible representation $\pi$ of a classical group, one says that it is supported by $X\cup\{\s\}$ 
if each irreducible subquotient of it is supported by that set.

\begin{definition}
Let 
$$
X=X_1\cup X_2
$$
be a partition of a selfcontragredient  $X\subseteq \mathcal C$. We shall say that this partition is regular if  $X_1$  is self contragredient\footnote{Then $X_2$ is also self contragredient}, and if  among $X_1$ and $X_2$ there is no reducibility, i.e. if 
$$
\rho\in X_1 \implies \nu\rho\not\in X_2.
$$
This is equivalent to say that $\rho_1\t\rho_2$ is irreducible for all $\rho_1\in X_1$ and $\rho_2\in X_2$.

For a partition $X=X_1\cup\dots \cup X_k$ we define to be regular in an analogous way.
\end{definition}

\begin{definition}
Let $\pi$ be a representation of $S_n$
supported in $X\cup\{\s\}$. Suppose that $X_1\cup X_2$ is a regular partition of a selfcontragredient  $X\subseteq \mathcal C$. Write $\mu^{\ast}(\pi)=\sum_i \b_i \otimes \g_i$, a sum
of irreducible representations in $R \otimes R[S]$. Let
$\mu_{X_1}^{\ast}(\pi)$ denote the sum of
every $\b_i \otimes \g_i$ in $\mu^{\ast}(\pi)$ such that the
support of $\b_i$ is contained in $X_1$
and the support of $\g_i$ is contained in
$X_2\cup\{\s\} $.

\end{definition}

Now we recall below the main results of \cite{J-supp}. As we have already mentioned, our presentation is slightly different from the presentation in \cite{J-supp}. In the rest of this section, $X_1\cup X_2$ will be a regular partition of a selfcontragredient  $X\subseteq \mathcal C$.

\begin{lemma} If $\pi$ has support contained in $X\cup\{\s\}$,
then $\mu_{X_1}^{\ast}(\pi)$ is nonzero.
\end{lemma}

\begin{definition}
Suppose $\b$ is a representation of a general linear group supported in
$X$.
Write $M^{\ast}(\b)=\sum_i \tau_i \otimes \tau'_i$, a sum of
irreducible representations in $R \otimes R$.
Let $M_{X_1}^{\ast}
(\b)$ denote the sum of every summand  $\tau_i \otimes \tau'_i$
in $M^{\ast}(\b)$ such that the support of $\tau_i$ is contained in
$X_1$ and the support of $\tau'_i$ is
contained in $X_2$.

\end{definition}

\begin{proposition}
Suppose $\b$ is a representation of a general linear group
with the support contained in $X$
and $\g$ a representation of $S_k$ with the support contained
in $X\cup\{\s\}$. 
Then,
\[
\mu_{X_1}^{\ast}(\b \rtimes \g)
=M_{X_1}^{\ast}(\b) \rtimes
\mu_{X_1}^{\ast}(\g).
\]

\end{proposition}

\begin{corollary}
\label{cor-mu-mod-X}
Suppose $\b$ has the support contained in $X_1$
and $\g$ has the support contained in $X_2\cup
\{\s \}$. Then
\begin{enumerate}
\item 
\[
\mu_{X_1}^{\ast}(\b \rtimes \g)=M^*_{GL}(\b) \otimes
\g.
\]

\item 
Write  
$$
s_{GL}(\g)=\Xi \otimes \s
$$
in the Grothendieck group\footnote{Clearly, $\Xi$ does not need to be irreducible.}.
Then
\[
\mu_{X_2}^{\ast}(\b \rtimes \g)
=\Xi \otimes \b\r\s.
\]

\end{enumerate}
\end{corollary}

\begin{definition}
\label{def}
Suppose $\pi$ is an irreducible representation of $S_n$ supported
in $X\cup\{\s\}$. Fix $i \in \{1, 2 \}$.
Then there exists an irreducible $\b_i \otimes \g_i$
with $\b_i$ supported on $X_{3-i}$ and $\g_i$ supported on
$X_i\cup\{\s\}$ such that
\[
\pi \hookrightarrow \b_i\rtimes \g_i.
\]
The representation 
$
\g_i
$
is uniquely determined by the above requirement, and it is denoted by
$$
X_i(\pi).
$$
Further, 
\begin{equation}
\label{right-factor}
\mu_{X_{3-i}}^{\ast}(\pi)
\leq \mu_{X_{3-i}}^{\ast}
        (\b_i \rtimes \g_i) 
=
\displaystyle
M^*_{GL}(\b_i)
\otimes \g_i.
\end{equation}

\end{definition}

Now we shall recall of the key theorem from the Jantzen's paper \cite{J-supp}:

\begin{theorem} (Jantzen) 
\label{main-Jantzen}
Suppose that $X_1\cup X_2$ is a regular partition of a selfcontragredient subset $X$ of $\mathcal C$,
and $\s$ an irreducible cuspidal representation of $S_r $.
Let $Irr(X_i; \s)$ denote the set of all irreducible
representations of all $S_n $, $n \geq 0$, supported on $X_i\cup\{\s\}$, and similarly for $Irr(X; \s).$

Then the map 
$$
Irr(X; \s)\longrightarrow Irr(X_1; \s) \times  Irr(X_2; \s),
$$
$$
\hskip10mm\pi\qquad \longmapsto \qquad (X_1(\pi), X_2(\pi))
$$
is a bijective correspondence. Denote the inverse mapping  by
$$
\Psi_{X_1,X_2}.
$$

\noindent
For  $\g_i \in Irr(X_i; \s)$ these  bijective correspondence have the following properties:
\begin{enumerate}

\item If $\g_i $
        is a representation of $S_{n_i+r} $,
        then 
        $$
        \pi=\Psi_{X_1,X_2}(\g_1, \g_2)
        $$
        is a representation of $S_{n_1+n_2+r} $

\item $\widetilde{\Psi_{X_1,X_2}(\g_1,\g_2)}
        =\Psi_{X_1,X_2}(\widetilde{\g_1}, \widetilde{\g_2})$
        and $X_i(\tilde{\pi})=\widetilde{X_i(\pi)}$,
        where $\tilde{\ }$ denotes contragredient.

\item ${\Psi_{X_1,X_2}(\g_1, \g_2)}^t
        =\Psi_{X_1,X_2}({\g_1}^t, {\g_2}^t)$
        and $X_i({\pi}^t)={X_i(\pi)}^t$,
        where $^t$ denotes the involution of  Aubert-Schneider-Stuhler .

\item Suppose that
\[
s_{GL}(\g_i)=\displaystyle{\sum_j}c_j(X_i)\tau_j(X_i)
\otimes \s,
\]
where $\tau_j(X_i)$ is an irreducible representation and
$c_j(X_i)$ its multiplicity.
Then 
\[
\begin{array}{l}
\mu_{X_i}^{\ast}(\Psi_{X_1,X_2}(\g_1, \g_2))
\\
\mbox{\ \ \ \ \ \ \ }
=\displaystyle{\sum_{j}}c_{j}(X_i) 
\tau_{j}(X_i) 
\otimes \g_{3-i}
\end{array}
\]

\item
Let $\b= \b(X_1)\times   \b(X_2)$
be an irreducible representation
of a general linear group with support of $\b(X_i)$ contained in $X_i$, $i=1,2$,
and $\Psi=\Psi_{X_1,X_2}(\g_1, \g_2)$
an irreducible representation of $S_k$ with support contained
in $X\cup\{\s\}$.
(We allow the possibility that $\b(X_i)=1$ or
$\g_i=\s$.) Suppose
\[
\b(X_i) \rtimes \g_i=
\sum_j m_j(X_i) \g_j(X_i; \s),
\]
with $\g_j(X_i;\s)$ irreducible and $m_j(X_i)$ its
multiplicity. Then,
\[
\b \rtimes \Psi
=\sum_{j_1, j_2} (m_{j_1}(X_1)  m_{j_2}(X_2))
\Psi_{X_1,X_2}(\g_{j_1}(X_1; \s), \g_{j_2}(X_2; \s)).
\]

\item
$\Psi_{X_1,X_2}(\g_1, \g_2)$ is
tempered (resp. square-integrable) if and only if
$\g_1, \g_2$ are both
tempered (resp. square-integrable).

\item
Suppose, in the subrepresentation setting in "tempered" formulation  of the Langlands classification,
\[
\g_i=L(\nu^{\alpha_1}\tau_1(X_i),
\dots, \nu^{\alpha_{\ell}} \tau_{\ell}(X_i); T(X_i; \s))
\]
for $i=1, 2$ (n.b. recall that $\tau_j(X_i)$ may be the
trivial representation of $GL(0,F)$; $T(X_i; \s)$ may just
be $\s$). Then,
\[
\begin{array}{rl}
\Psi_{X_1,X_2}(\g_1,\g_2)
&=L(\nu^{\alpha_1}\tau_1(X_1)
\times  \nu^{\alpha_1}\tau_1(X_2), 
\\
& \dots, \nu^{\alpha_{\ell}}\tau_{\ell}(X_1) \times
\nu^{\alpha_{\ell}}\tau_{\ell}(X_2);
\Psi_{X_1,X_2}(T(X_1; \s),  T(X_2; \s))).
\end{array}
\]
In the other direction, if
\[
\pi=L(\nu^{\alpha_1}\tau_1(X_1)
 \times \nu^{\alpha_1}\tau_1(X_2),\dots, \nu^{\alpha_{\ell}}\tau_{\ell}(X_1) \times\nu^{\alpha_{\ell}}\tau_{\ell}(X_2);
T(X; \s)),
\]
then
\[
X_i(\pi)=L(\nu^{\alpha_1}\tau_1(X_i), \dots, \nu^{\alpha_{\ell}}
\tau_{\ell}(X_i); X_i(T(X; \s))).
\]
(In the quotient setting of the Langlands classification, the same
results hold.)

\item
Suppose,
\[
\mu^{\ast}(\g_i)=\sum_{j} n_{j}(X_i)
\eta_{j}(X_i) \otimes \theta_{j}(X_i; \s),
\]
with $\eta_{j}(X_i) \otimes \theta_{j}(X_i; \s)$
irreducible and $n_{j}(X_i)$ its multiplicity. Then,
\[
\begin{array}{l} \hskip10mm
\mu^{\ast}(\Psi_{X_1,X_2}(\g_1,\g_2)) \\
\mbox{\ \ \ \ \ \ \ \ \ }
=\displaystyle{\sum_{j_1, j_2}} (n_{j_1}(X_1)  n_{j_2}(X_2))
(\eta_{j_1}(X_1) \times  \eta_{j_2}(X_2)) \otimes
\Psi_{X_1,X_2}(\theta_{j_1}(X_1; \s),  \theta_{j_2}(X_2; \s)).
\end{array}
\]

\item Let $X=X_1\cup X_2\cup X_3$ be a regular partition and $\pi\in Irr(X;\s)$. Then
$$
X_1\big((X_1\cup X_2)(\pi)\big)=X_1\big((X_1\cup X_3)(\pi)\big).
$$
In the other direction we have
$$
\Psi_{X_1\cup X_2,X_3}\big(\Psi_{X_1, X_2}(\pi_1,\pi_2),\pi_3\big)
=
\Psi_{X_1, X_2\cup X_3}\big(\pi_1,\Psi_{X_2, X_3}(\pi_2,\pi_3)\big)
$$
for $\pi_i\in Irr(X_i;\s)$.

\end{enumerate}

\end{theorem}

\begin{remark} 
\begin{enumerate}
\item Let $\beta_i$ be an irreducible representation  of a general linear group supported in $X_i$,  $i=1,2$, and let $\gamma_i$ be an irreducible representation  of a classical $p$-adic  group supported in $X_i\cup\{\s\}$,  $i=1,2$. Then (5) of the above theorem implies
$$
(\beta_1\t\b_2)\r \Psi_{X_1,X_2}(\g_1,\g_2)
\text{ is irreducible }\iff\text{ both $ \beta_i\r\gamma_i$ are irreducible}.
$$
\item
One can express the above theorem without the last claim, in a natural way for a regular partition in more than two pieces.
\end{enumerate}
\end{remark}

\section{Cuspidal lines}
    \label{lines}
\setcounter{equation}{0}

Let $\rho$ be an irreducible unitarizable cuspidal representation of a general linear group. Denote
$$
X_\rho=\{\nu^x\rho;x\in\R\}\cup \{\nu^x\tilde\rho;x\in\R\},
$$
$$
X_\rho^c=\mathcal C\backslash X_\rho.
$$

For an irreducible representation $\pi$ of a classical $p$-adic group take any finite set of different classes $\rho_1,\dots,\rho_k\in\mathcal C_u$ such that $\rho_i\not\cong\rho_j$ for any $i\ne j$, and that $\pi$ is supported in
$$
X_{\rho_1}\cup\dots\cup X_{\rho_k} \cup \{\s\}.
$$
Then $\pi$ is uniquely determined by
$$
(X_{\rho_1}(\pi),\dots, X_{\rho_k}(\pi)).
$$
Now we have a natural

{\bf Preservation of unitarizability question:}  {\it Let $\pi$ be an irreducible weakly  real 
representation of a classical $p$-adic group\footnote{We do not need to assume  $\pi$ to be  weakly real in the above question. Theorem \ref{unitary-red-1} (or \ref{unitary-red-2}) implies that this is an equivalent to the above question.}. 
 Is $\pi$ unitarizable if and only if all $X_{\rho_i}(\pi)$ are unitarizable?}

\section{Proof of the main result}
\label{proof main theo}
\setcounter{equation}{0}

\begin{theorem}
\label{th-main} Suppose that $\theta$ is an irreducible unitarizable representation of a classical group, and suppose that the  infinitesimal character of some $X_\rho(\theta)$ is the same as the infinitesimal character of a generalized Steinberg representation supported in $X_\rho\cup\{\s\}$ with $\a_{\rho,\s}\in\frac12\Z$\footnote{As we already noted, this is know if char$(F)=0$.}. Then $X_\rho(\theta)$ is 
 the generalized Steinberg representation, or its  Aubert-Schneider-Stuhler  dual.

 In particular, if char$(F)=0$, then $X_\rho(\theta)$ is unitarizable.
\end{theorem}

\begin{proof} 
Denote
$$
\theta_\rho=X_\rho(\theta),\qquad \theta_\rho^c=X_\rho^c(\theta).
$$
Then
$$
\theta=\Psi_{X_\rho,X_\rho^c}(\theta_\rho,\theta_\rho^c).
$$
Suppose that $\theta_\rho$ is neither the generalized Steinberg representation, nor it is its  Aubert-Schneider-Stuhler  dual.  Now Proposition \ref{prop-main}  implies that there exists a selfcontragredient unitarizable representation $\pi$ of a general linear group supported in $X_\rho$ such that the length of  
$$
\pi\r\theta_\rho
$$
 is at least 5, and that the multiplicity of $\pi\o\theta_\rho$ in the Jacquet module of $\pi\r\theta_\rho$ is at most 4.

Consider now 
$$
\pi\r\theta= \pi\r\Psi_{X_\rho,X_\rho^c}(\theta_\rho,\theta_\rho^c).
$$
Then this representation is of length $\geq 5$
(take in (5) of Jantzen theorem $\b(X_\rho)=\pi,\b(X_\rho^c)=1$, and multiply it  by the representation $\Psi_{X_\rho,X_\rho^c}(\theta_\rho,\theta_\rho^c)$).

We shall now use the assumption  that $\theta=\Psi_{X_\rho,X_\rho^c}(\theta_\rho,\theta_\rho^c)$ is unitarizable. From the fact that the length of $\pi\o \Psi_{X_\rho,X_\rho^c}(\theta_\rho,\theta_\rho^c)$ is at least 5 and the exactness of the Jacquet module functor, it follows that the multiplicity of $\pi\o \Psi_{X_\rho,X_\rho^c}(\theta_\rho,\theta_\rho^c)$ in $\mu^*(\pi\r \Psi_{X_\rho,X_\rho^c}(\theta_\rho,\theta_\rho^c))$ is at least five.

By the definition of $\theta_\rho$, we can chose an irreducible representation $\phi$ of a general linear group  supported in $X_\rho^c$ such that 
$$
\Psi_{X_\rho,X_\rho^c}(\theta_\rho,\theta_\rho^c)\h\phi\r \theta_\rho.
$$
By the Frobenius reciprocity, $\phi\o \theta_\rho$ is a sub quotient of the Jacquet module of  $\Psi_{X_\rho,X_\rho^c}(\theta_\rho,\theta_\rho^c)$. Denote its multiplicity  by $k$. This implies that the multiplicity of $\pi\o \phi\o\theta_\rho$ in $\mu^*(\pi\r \Psi_{X_\rho,X_\rho^c}(\theta_\rho,\theta_\rho^c))$ is at least $5k$.

Recall that the support of $\pi$ is in $X_\rho$ and the support of $\phi$ is in $X_\rho^c$.
Let $\Pi$ be an irreducible representation of a general linear group which has in its Jacquet module  $\pi\o \phi$. Then $\Pi\cong \pi'\t\phi'$, where the support of $\pi'$ is in $X_\rho$ and the support of $\phi'$ is in $X_\rho^c$. Further, $\pi$ and $\pi'$ are representations of the same group (as well as $\phi$ and $\phi'$). Frobenius reciprocity implies that $\pi'\o\phi'$ is in the Jacquet module of $\Pi$.
Further, the formula $m^*(\Pi)=m^*(\pi)\t m^*(\phi)$ implies that if we have in the Jacquet module of $\Pi$ an irreducible representation of the form $\pi''\o\phi''$, where the support of $\pi''$ is in $X_\rho$ and the support of $\phi''$ is in $X_\rho^c$, then $\pi''\cong\pi'$, $\phi''\cong\phi'$, and the multiplicity of $\pi'\t\phi'$ in the Jacquet module of $\Pi$ is one. 

This first implies that $\Pi\cong\pi\t\phi$, then that 
 the only irreducible representation of a general linear group which has in its Jacquet module $\pi\o \phi$ is $\pi\t \phi$, and further  that the multiplicity  
of $\phi\o\pi$ in the Jacquet module  $\pi\t \phi$ is  one. 

This and  the transitivity of the Jacquet modules imply that the multiplicity of $\phi\o \pi\o \theta_\rho$ in the Jacquet module of $\mu^*(\pi\r \Psi_{X_\rho,X_\rho^c}(\theta_\rho,\theta_\rho^c))$ is at least $5k$.

Now we examine in a different way the multiplicity of $\phi\o \pi\o \theta_\rho$ in the Jacquet module of $\mu^*(\pi\r \Psi_{X_\rho,X_\rho^c}(\theta_\rho,\theta_\rho^c))$.
Observe that $\phi\o \pi\o \theta_\rho$ must be a sub quotient of a Jacquet module of the following part
$$
\mu^*_{X_\rho^c}(\pi\r \Psi_{X_\rho,X_\rho^c}(\theta_\rho,\theta_\rho^c))=(1\o \pi)\r \mu^*_{X_\rho^c}(\Psi_{X_\rho,X_\rho^c}(\theta_\rho,\theta_\rho^c))
$$
of $\mu^*(\pi\r \Psi_{X_\rho,X_\rho^c}(\theta_\rho,\theta_\rho^c))$.
Recall that by \eqref{right-factor}, $\mu^*_{X_\rho^c}(\Psi_{X_\rho,X_\rho^c}(\theta_\rho,\theta_\rho^c))$ is of the form $*\o\theta_\rho$. If we want to get $\phi\o \pi\o \theta_\rho$ from a term from here, it must be $\phi\o\theta_\rho$. Recall that we have this term with multiplicity $k$ here. Therefore, we need to see the multiplicity of   $\phi\o \pi\o \theta_\rho$ in the Jacquet module of $k\cdot (1\o\pi)\r(\phi\o\theta_\rho)=k\cdot  (\phi\o\pi\r\theta_\rho)$. We know that this multiplicity is at most $4k$. Therefore, $5k\leq4k$ (and $k\geq1$).
This is a contradiction. 

Therefore, $\theta_\rho$ is  
the generalized Steinberg representation or its  Aubert-Schneider-Stuhler  dual. The generalized Steinberg representation is unitarizable (since it is square integrable). Further in characteristic zero, \cite{H} (or \cite{Moe-mult-1} and \cite{MoeR}) implies that its  Aubert-Schneider-Stuhler  dual is unitarizable. Therefore, $\theta_\rho$ is unitarizable if we are in the characteristic zero.
\end{proof}

\section{Irreducible generic and irreducible unramified representations}
\label{generic}
\setcounter{equation}{0}

We consider in this section  quasi-split classical $p$-adic groups (see \cite{LMuT} for more details).
One can find in \cite{LMuT} more detailed exposition of the facts about irreducible generic representations and unitarizable subclasses that we shall use here. We shall recall here only very briefly of some of that facts.

Let $\g$ be an irreducible  representation of a classical group. 
Let $X_1\cup X_2$ be a regular partition of $\mathcal C$.
Now \cite{R-Wh} directly implies that  $\g$ is generic if and only if  $X_1(\g)$ and $X_2(\g)$ are generic.
Therefore, 
\begin{equation}
\label{gen}
\text{$\g$ is generic if and only if all $X_\rho(\g) $ are generic, $\rho\in\mathcal C_u$.}
\end{equation}
Analogous statement holds for temperness by (6) of Theorem \ref{main-Jantzen}.

Recall that  by (5) of Theorem \ref{main-Jantzen}, if the support of some irreducible representation $\beta$ of a general linear group is contained in $X_{\rho'}$,
then holds
\begin{equation}
\label{irr}
\b\r\g \text{ is irreducible } \iff \b\rtimes X_{\rho'}(\g) \text{ is irreducible}.
\end{equation}

Denote by $\mathcal C_u'$ any subset of $\mathcal C_u$ satisfying:
 $$
 \mathcal C_u'\cup (\mathcal C_u')\tilde{\ }=\mathcal C_u \text{ and  }
\rho \in \mathcal C_u'\cap (\mathcal C_u')\tilde{\ } \implies\rho\cong\tilde\rho.
$$

Let $\pi $ be an irreducible generic representation of a classical group. We can write
$\pi$ uniquely as
\begin{equation} \label{lang1}
\pi\cong\delta_1\times\dots\times\delta_k\rtimes\tau
\end{equation}
where the $\delta_i$'s are irreducible essentially square-integrable representations  of
general linear groups which satisfy
\begin{equation}\label{lang2}
e(\delta_1)\ge  \cdots \ge e(\delta_k)>0,
\end{equation}
and  $\tau$ is a generic irreducible tempered  representation of a
classical group.

For $\rho'\in\mathcal C_u'$ chose some irreducible representation $\Gamma_{\rho'}^c$ of a general linear group such that  
$$
\tau\h \Gamma_{\rho'}^c\rtimes X_{\rho'}(\tau),
$$
and that $\Gamma_{\rho'}^c$ is supported out of $X_{\rho'}$.
Observe that
$$
\pi\cong \bigg(\prod_{\rho\in\mathcal C'_u}\bigg(\prod_{\supp(\d_i)\subseteq X_\rho} \delta_i\bigg)\bigg)\rtimes\tau
\h
 \bigg(\prod_{\rho\in\mathcal C'_u}\bigg(\prod_{\supp(\d_i)\subseteq X_\rho} \delta_i\bigg)\bigg)
 \t
 \Gamma_{\rho'}^c\rtimes X_{\rho'}(\tau)
$$
$$
\cong
 \bigg(\prod_{\rho\in\mathcal C'_u\backslash\{\rho'\}}\bigg(\prod_{\supp(\d_i)\subseteq X_\rho} \delta_i\bigg)\bigg)
 \t
 \Gamma_{\rho'}^c
 \t \bigg(\prod_{\supp(\d_i)\subseteq X_{\rho'}} \delta_i\bigg)\rtimes X_{\rho'}(\tau).
$$
One easily sees that there exists an irreducible sub quotient $\Pi_{\rho'}^c$ of 
$$
 \bigg(\prod_{\rho\in\mathcal C'_u\backslash\{\rho'\}}\bigg(\prod_{\supp(\d_i)\subseteq X_\rho} \delta_i\bigg)\bigg)
 \t
 \Gamma_{\rho'}^c
 $$
  such that
 $$
 \pi\h \Pi_{\rho'}^c\t \bigg(\prod_{\supp(\d_i)\subseteq X_{\rho'}} \delta_i\bigg)\rtimes X_{\rho'}(\tau).
 $$
Since $\Pi_{\rho'}^c$ is supported out of $X_{\rho'}$ and $\bigg(\prod_{\supp(\d_i)\subseteq X_{\rho'}} \delta_i\bigg)\rtimes X_{\rho'}(\tau)$ is irreducible and supported in $X_{\rho'}\cup\{\s\}$, we get that
\begin{equation}
\label{comp}
X_{\rho'}(\pi)=\bigg(\prod_{\supp(\d_i)\subseteq X_{\rho'}} \delta_i\bigg)\rtimes X_{\rho'}(\tau).
\end{equation}

 Let $\pi\cong\delta_1\times\dots\times\delta_k\rtimes\tau$ be as in \eqref{lang1}. Then for any square-integrable representation $\delta$ of a general linear group
denote by $\mathcal E_\pi(\delta)$ the multiset of exponents
$e(\delta_i)$ for those $i$ such that $\delta_i^u\cong\delta$. We denote below by $\mathbf 1_G$ the trivial one-dimensional representation of a group $G$. Now we recall of the solution
of the unitarizability problem for irreducible generic representations of classical $p$-adic groups obtained in \cite{LMuT}.

\begin{theorem} \label{generic-th}
Let $\pi$ be given as in \eqref{lang1}. Then $\pi$  is unitarizable
if and only if for all irreducible square integrable representations $\delta$ of general linear groups hold
\begin{enumerate}
\item \label{hermit}
$\mathcal E_\pi(\tilde{\delta})=\mathcal E_\pi(\delta)$, i.e. $\pi$ is
Hermitian.
\item \label{nongntype} If either $\delta\not\cong\tilde{\delta}$ or
$\nu^\frac12\delta\rtimes\mathbf 1_{G_0}$ is reducible then $0<\alpha<\frac12$
for all $\alpha\in \mathcal E_\pi(\delta)$.
\item \label{gntype} If $\tilde{\delta}\cong\delta$ and $\nu^\frac12\delta\rtimes\mathbf 1_{G_0}$ is
irreducible then $\mathcal E_\pi(\delta)$ satisfies Barbasch'
conditions, i.e. we have
$\mathcal E_\pi(\delta)=\{\alpha_1,\dots,\alpha_k,\beta_1,\dots,\beta_l\}$
with
\[
0<\alpha_1\le\dots\le\alpha_k\le\frac12<\beta_1<\dots<\beta_l<1
\]
such that
\begin{enumerate}
\item \label{irreda} $\alpha_i+\beta_j\ne1$ for all $i=1,\dots,k$,
$j=1,\dots,l$; $\alpha_{k-1}\ne\frac12$ if $k>1$.
\item \label{irredb} $\#\{1\le i\le k:\alpha_i>1-\beta_1\}$ is
even if $l>0$.
\item \label{irredc} $\#\{1\le i\le k:1-\beta_j>\alpha_i>1-\beta_{j+1}\}$
is odd for $j=1,\dots,l-1$.
\item \label{tau} $k+l$ is even if $\delta\rtimes\tau$ is reducible.
\end{enumerate}
\end{enumerate}
\end{theorem}

Observe that \eqref{irr} implies that if $\supp(\d_i)\subset X_{\rho'}$, then
\begin{equation}
\label{irr+}
\d_i\r\tau \text{ is irreducible } \iff \d_i\rtimes X_{\rho'}(\tau) \text{ is irreducible}.
\end{equation}

Let $\pi$ be a generic representation. We can then present it  by the formula \eqref{lang1}

Suppose that $\pi$ is unitarizable. This implies that $\pi$ satisfies the above theorem. Now from \eqref{irr+}, the above theorem implies that
$
X_{\rho'}(\pi)\cong\bigg(\prod_{\supp(\d_i)\subseteq X_{\rho'}} \delta_i\bigg)\rtimes X_{\rho'}(\tau)
$
is unitarizable (we need \eqref{irr+} only for (d) of (3) in the above theorem).

Suppose now that all $X_{\rho'}(\pi)=\bigg(\prod_{\supp(\d_i)\subseteq X_{\rho'}} \delta_i\bigg)\rtimes X_{\rho'}(\tau)$, $\rho\in\mathcal C_u'$, are unitarizable. Then each of them satisfy the above theorem. Now the above theorem and \eqref{irr+} imply that $\pi$ is unitarizable.

Therefore, we have proved the following

\begin{corollary} For an irreducible generic representation $\pi$ of a classical group holds
$$
\pi\text{ is unitarizable }\iff\text{ all $ X_{\rho}(\pi)$, $\rho\in\mathcal C_u$, are unitarizable}.\qed
$$
\end{corollary}

In a similar way, using the classification of the irreducible unitarizable unramified  representations of split classical $p$-adic  groups  in \cite{MuT} (or as it is stated in \cite{T-auto}), we get that the above fact holds for irreducible  unramified  representations of classical $p$-adic groups.

\section{Question of independence}
\label{independence}
\setcounter{equation}{0}

Let $\rho$ and $\s$ be irreducible unitarizable cuspidal representations  of a general linear and a classical group respectively. If there exists a non-negative $\a_{\rho,\s}\in \frac12\Z$ such that
$$
\nu ^{\a}\rho\r\s
$$
reduces. When we fix $\rho$ and $\s$ as above, to shorten notation then    this $\a_{\rho,\s}$ will be denoted also by $\a$.

By a $\Z$-segment in $\R$  we shall mean a subset of form $\{x,x+1,\dots,x+l\}$
of $\R$. 
We shall denote this subset  by $[x,x+l]$. For such a segment $\D$, we denote
$$
\D^{(\rho)}=\{\nu^x\rho;x\in\D\}.
$$

We shall fix two pairs $\rho_i,\s_i$ as above, such that
$$
\a_{\rho_1,\s_1}=\a_{\rho_2,\s_2}
$$
and denote it by
$$
\a.
$$ 
We shall construct a natural bijection
$$
E_{1,2}: Irr(X_{\rho_1};\s_1)\rightarrow Irr(X_{\rho_2};\s_2),
$$
which will be canonical, except in the case when $\a=0$.
First we shall define $E_{1,2}$ on the irreducible square integrable representations. 

A classification of irreducible square integrable representations of classical $p$-adic groups modulo cuspidal data is completed in \cite{Moe-T}. We shall freely use notation of that paper, and also of \cite{T-Sing}.  We shall very briefly recall of parameters of irreducible square integrable representations in $Irr(X_\rho;\s)$ (one can find more details in \cite{T-Sing}, sections 16 and 17). Below $(\rho,\s)$ will denote $(\rho_1,\s_1)$ or $(\rho_2,\s_2)$.

An irreducible square integrable representation $\pi\in Irr(X_\rho;\s)$ is parameterized by Jordan blocks $Jord_\rho(\pi)=\{\D_1^{(\rho)},\dots,\D_k^{(\rho)}\}$, where $\D_i$ are $\Z$-segments contained  in $\a+\Z$, and by a partially defined function $\epsilon_\rho(\pi)$ (partial cuspidal support is $\s$). Since $\{\D_1^{(\rho)},\dots,\D_k^{(\rho)}\}$ and $\{\D_1,\dots,\D_k\}$ are in a natural bijective correspondence, we can view $\epsilon_\rho(\pi)$ as defined (appropriately) on $\{\D_1,\dots,\D_k\}$ (which means that $\epsilon_\rho(\pi)$ is independent of particular $\rho$). In  sections 16 and 17  of \cite{T-Sing}, it  is explained how $\pi$ and the triple
$$
(\{\D_1^{(\rho)},\dots,\D_k^{(\rho)}\},\epsilon_\rho(\pi),\s)
$$
are related. In this case we shall write
\begin{equation}
\label{si}
\pi\longleftrightarrow (\{\D_1^{(\rho)},\dots,\D_k^{(\rho)}\},\epsilon_\rho(\pi),\s).
\end{equation}

Take irreducible square integrable representations $\pi_i\in Irr(X_\rho;\s)$, $i=1,2$. Suppose
\begin{equation}
\label{si+}
\pi_1\longleftrightarrow (\{\D_1^{(\rho_1)},\dots,\D_k^{(\rho_1)}\},\epsilon_{\rho_1}(\pi_1),\s_1).
\end{equation}
 Then we define
$$
E_{1,2}(\pi_1)=\pi_2
$$
if 
$$
\pi_2\longleftrightarrow (\{\D_1^{(\rho_2)},\dots,\D_k^{(\rho_2)}\},\epsilon_{\rho_1}(\pi_1),\s_2).
$$
For defining $E_{1,2}$ on the whole $Irr(X_{\rho_1};\s_1)$, the key step is an extension of  $E_{1,2}$ from the square integrable classes to the tempered classes. For this, we shall use parameterization of irreducible tempered representations obtained in \cite{T-temp}\footnote{Another possibility would be to use the Jantzen's parameterization obtained in \cite{J-temp} (we do not know if using \cite{J-temp} would result with the same mapping $E_{1,2}$). }.

Let $\pi\in Irr(X_\rho,;\s)$ be square integrable and 
let  $\d:=\d(\D^{(\rho)})$ be an irreducible (unitarizable) square integrable representation of a general linear group, where $\D$ is a segment in $\a+\Z$ such that $\d\r\pi$ reduces (one directly reads from the invariants \eqref{si} when this happens).
Now Theorem 1.2 of \cite{T-temp} defines the irreducible tempered subrepresentation $\pi_\d$ of $\d\r\pi$. The other irreducible summand is denoted by $\pi_{-\d}$.

Let $\pi\in Irr(X_\rho,;\s)$ be square integrable, 
let  $\d_i:=\d(\D_i^{(\rho)})$ be different  irreducible (unitarizable) square integrable representations of general linear groups, where $\D_i$ are $\Z$-segments contained  in $\a+\Z$ such that all $\d_i\r\pi$ reduce, and let $j_i\in\{\pm1\}$, $i=1,\dots,n$. Then there exists a unique (tempered) irreducible representation $\pi'$ of a classical group  such that
$$
\pi'\h\d_1\t\dots\t\d_{i-1}\t\d_{i+1}\t\dots\t\d_n\r\pi_{j_i\d_i},
$$
for all $i$. Then we denote
$$
\pi'=\pi_{j_1\d_1,\
\dots,j_n\d_n}.
$$
In the situation as above 
we define
$$
E_{1,2}(\pi_{j_1\d(\D_1^{(\rho_1)}),\
\dots,j_n\d(\D_1^{(\rho_1)})})=
E_{1,2}(\pi)_{j_1\d(\D_1^{(\rho_2)}),\
\dots,j_n\d(\D_1^{(\rho_2)})}.
$$
Let additionally $\G_1^{(\rho)},
\dots,\G_m^{(\rho)}$ be segments of cuspidal representations  such that  for each $i$,  either $\G_i$ is among $\D_j$'s, or $\d(\G_i^{(\rho)})\r\pi$ is irreducible, and $-\G_i=\G_i$. Then the tempered representation 
\begin{equation}
\label{temp}
\d(\G_1^{(\rho)})\t\dots\t\d(\G_m^{(\rho)})\r 
\pi_{j_1\d(\D_1^{(\rho)}),\
\dots,j_n\d(\D_1^{(\rho)})}
\end{equation}
is irreducible. We define
$$
E_{1,2}(\d(\G_1^{(\rho_1)})\t\dots\t\d(\G_m^{(\rho_1)})\r \pi_{j_1\d(\D_1^{(\rho_1)}),\
\dots,j_n\d(\D_1^{(\rho_1)})})
=
\hskip60mm
$$
$$
\hskip60mm
\d(\G_1^{(\rho_2)})\t\dots\t\d(\G_m^{(\rho_2)})\r E_{1,2}(\pi_{j_1\d_1,\
\dots,j_n\d_n}).
$$
In this way we have define $E_{1,2}$ on the subset of all the tempered classes  in $Irr(X_{\rho_1};\s)$.

Let now $\pi$ be any element of $Irr(X_{\rho_1};\s)$. Write
$$
L(\D_1^{(\rho_1)},\dots,\D_k^{(\rho_1)};\tau)
$$
as a Langlands quotient ($\D_i$ are $\Z$ segments in $\R$ and $\tau$ is a tempered class in $Irr(X_{\rho_1};\s)$). Then we define
$$
E_{1,2}(L(\D_1^{(\rho_1)},\dots,\D_k^{(\rho_1)};\tau))
=
L(\D_1^{(\rho_2)},\dots,\D_k^{(\rho_2)};E_{1,2}(\tau)).
$$

{\bf Independence  of unitarizability   question:} {\it Let $\rho_1,\rho_2,\s_1$ and $\s_2$ be irreducible cuspidal representations as above. Suppose $\a_{\rho_1,\s_1}=\a_{\rho_2,\s_2}$ and construct the mapping $E_{1,2} $ as above. Let $\pi \in Irr(X_{\rho_1};\s)$. Is $\pi$ unitarizable if and only if $E_{1,2}(\pi)$ is unitarizable?}

One can also ask if the other important representation theoretic data are preserved by $E_{1,2}$ (Jacquet modules, irreducibilities of parabolically induced representations, Kazhdan-Lusztig multiplicities etc.).

\begin{remark} In this remark we consider irreducible  generic representations, and assumptions on the groups are the same as in the section \ref{generic}. We continue with the previous notation. Let $\d:=\d(\D^{(\rho)})$be an irreducible (unitarizable) square integrable representation of a general linear group, where $\D$ is a segment in $\a+\Z$. 

Then we know that $\nu^{1/2}\d(\D^{(\rho)})\r\mathbf 1_{S_{0}}$ reduces if and only if 
\begin{enumerate}

\item card$(\D)$ is odd if $\a\not\in\Z$;

\item card$(\D)$ is even if $\a\in \Z$.
\end{enumerate}

Therefore the conditions of reducibility of $\nu^{1/2}\d(\D^{(\rho)})\r\mathbf 1_{S_0}$ in (2) and irreducibility in (3) of Theorem \ref{main} does not depend on $\rho$, but only on $\D$ and $\a$.

Further, let $\tau$ be the representation in \eqref{temp}. Now $\d(\D^{(\rho)})\r\tau$ is reducible if and only if

\begin{enumerate}

\item[(i)] $\a\in\D$;

\item[(ii)] $\D\not\in\{\D_1,\dots,\D_n\}$ (recall that $\D_1,\dots\D_n$ form the Jordan block of $\pi$ along $\rho$);

\item $\D\not\in\{\G_1,\dots,\G_m\}$.

\end{enumerate}

Obviously, these conditions again does not depend on $\rho_i$, but on $\a=\a_{\rho,\s}$ and the parameters which are preserved by $E_{1,2}$. Therefore now Theorem \ref{main} implies that the  above Independence  of unitarizability   question   has positive answer for the  irreducible generic representations, i.e. the unitarizability in this case does not depend on particular $\rho$ and $\s$, but only on $\a=\a_{\rho,\s}$.

\end{remark}

\section{The case of unitary groups}
\label{unitary}
\setcounter{equation}{0}

We shall now comment the case of unitary groups. Then we have a quadratic extension $F'$ of $F$ and the non-trivial element $\Theta$ of the Galois group. Let $\pi$ be a representation of $GL(n,F')$. Then the representation $g\mapsto \tilde \pi(\Theta(g))$ will be called the $F'/F$-contragredient of $\pi$ and denoted by $\check{\pi}$.

The results of this paper hold also for the unitary groups if we replace everywhere representations of general linear groups over $F$ by representations of general linear groups over $F'$, and contragredients of representations of general linear groups by $F'/F$-contragredients.
The proofs of the statements are the same (with one exception - see below), after we apply the above two changes everywhere. 

The only difference is that the unitarizability of the Aubert-Schneider-Stuhler involution of the generalized Steinberg representation we do not get from \cite{H}. The unitarizability of the Aubert-Schneider-Stuhler involution of a general irreducible square integrable representation of a classical group over a $p$-adic field of  characteristic zero follows  from \cite{Moe-mult-1}, \cite{Moe-ASS} (proposition 4.2 there)
 and (in the non-quasi split case) \cite{MoeR} (Theorem 4.1 there).

\end{document}